\documentclass[12pt,leqno]{amsart}
\usepackage{amssymb,mathrsfs,euscript}
\textheight9in  \def\DATE{\today}
\newtheorem{theorem}{Theorem}
\newtheorem{definition}[theorem]{Definition}
\newtheorem{fact}[theorem]{Fact}

\newtheorem{example}[theorem]{Example}
\newtheorem{lemma}[theorem]{Lemma}

\newtheorem{proposition}[theorem]{Proposition}
\newtheorem{remark}[theorem]{Remark}
\newtheorem*{examples}{Examples}

\catcode`\@=11
\def\@evenfoot{\rule{0pt}{20pt}[\today] \hfill}
\def\@oddfoot{\rule{0pt}{20pt}\hfill [\today]}

\def\jaruska #1|#2|#3{
\put(0.00,0.00){\makebox(0.00,0.00){$\bullet$}}
\put(3.50,-1.00){\makebox(0.00,0.00){$#3$}}
\put(0.00,0.00){\line(0,1){1.00}}
\put(0.00,0.00){\line(0,-1){1.00}}
\put(0.00,00){\line(1,-1){1.00}}
\put(0.00,00){\line(2,-1){2}}
\put(0.00,00){\line(-1,-1){1.00}}
\put(0.00,00){\line(-2,-1){2}}
\put(#1,-1){
\put(0.00,0.00){\makebox(0.00,0.00){$\bullet$}}
\put(0.00,0.00){\line(0,-1){1.00}}
\put(0.00,00){\line(1,-1){1.00}}
\put(0.00,00){\line(2,-1){2}}
\put(0.00,00){\line(-1,-1){1.00}}
\put(0.00,00){\line(-2,-1){2}}
\put(#2,-1){
\put(0.00,0.00){\makebox(0.00,0.00){$\bullet$}}
\put(0.00,0.00){\line(0,-1){1.00}}
\put(0.00,00){\line(1,-1){1.00}}
\put(0.00,00){\line(2,-1){2}}
\put(0.00,00){\line(-1,-1){1.00}}
\put(0.00,00){\line(-2,-1){2}}
}}}

\def\R#1{\boxed{#1}\,}
\def\tri{\hbox{$\bullet\!\bullet\!\bullet$}}

\def\dve{\hbox{$\bullet\bullet$}}
\def\b{\bullet}\def\KK{{\mathcal K}}
\def\hustykrouzek{
\bezier{300}(-2,0)(-2,.5)(-1.73,1)
\bezier{300}(-1.73,1)(-1.48,1.48)(-1,1.73)
\bezier{300}(-1,1.73)(-.5,2)(0,2)
\bezier{300}(2,0)(2,.5)(1.73,1)
\bezier{300}(1.73,1)(1.48,1.48)(1,1.73)
\bezier{300}(1,1.73)(.5,2)(0,2)
\bezier{300}(2,0)(2,-.5)(1.73,-1)
\bezier{300}(1.73,-1)(1.48,-1.48)(1,-1.73)
\bezier{300}(1,-1.73)(.5,-2)(0,-2)
\bezier{300}(-2,0)(-2,-.5)(-1.73,-1)
\bezier{300}(-1.73,-1)(-1.48,-1.48)(-1,-1.73)
\bezier{300}(-1,-1.73)(-.5,-2)(0,-2)
}

\def\exepttree{\hskip .5em\raisebox{-.2em}{\rule{.8pt}{1.1em}}  \hskip .5em}
\def\card{{\rm card}}\def\calC{{\mathcal{C}}}\def\rada#1#2{#1,\ldots,#2}
\def\Rada#1#2#3{#1_{#2},\dots,#1_{#3}}\def\Dbar{{\sf D}}
\def\epi{ \twoheadrightarrow}\def\Sree#1{{\EuScript{S}^n_{#1}}}
\def\Sreep#1{{\EuScript{S}^3_{#1}}}
\def\Free{{\pAss^n_d(V)}}\def\Tree#1{{\EuScript{T}^n_{#1}}}
\def\q{{\hbox{$\bullet$}}}\def\qq{{\hbox{\small$\bullet$}}}
\def\dst{\delta_{\rm st}}\def\Cst#1{{C_{\sstildeAss}^{#1}(A;A)_{\rm st}}}
\def\PCst#1{{C_{\calP}^{#1}(A;A)_{\rm st}}}\def\sspAss{{p\ssAss}}
\def\PC#1{{C_{\calP}^{#1}(A;A)}}\def\sstAss{{t\ssAss}}
\def\C#1{{C^{#1}_{\sstildeAss}(A;A)}}
\def\scrR{{\mathscr R}}\def\pa{{\partial}}\def\Hom{{\it Hom\/}}
\def\H#1{{H_{\sstildeAss}^{#1}(A;A)}}
\def\Hst#1{{H_{\sstildeAss}^{#1}(A;A)_{\rm st}}}
\def\PHst#1{{H_{\calP}^{#1}(A;A)_{\rm st}}}
\def\PH#1{{H_{\calP}^{#1}(A;A)}}
\def\bbbC{{\mathbb C}}\def\frakZ{{\mathfrak Z}}\def\bbbR{{\mathbb R}}

\def\otexp#1#2{{#1}^{\otimes #2}}\def\ot{\otimes}
\def\ssAss{\hbox{\tiny $\mathcal{A}\it ss$}}
\def\Ass{\hbox{$\mathcal{A}\it ss$}}
\def\bfk{{\mathbf k}}\def\tAss{{t\Ass}}
\def\tildeAss{\widetilde{\Ass}}
\def\sstildeAss{\widetilde{\ssAss}}\def\ssttildeAss{t\widetilde{\ssAss}{}}
\def\id{\hbox{$1 \hskip -.3em 1$}}\def\ssptildeAss{p\widetilde{\ssAss}{}}
\def\calP{{\mathcal P}}\def\Span{{\it Span}}\def\pAss{{p\Ass}}
\def\osusp{\hbox{\bf s\hskip .1em}}
\def\ttildeAss{{t\tildeAss}{}}\def\ptildeAss{{p\tildeAss}{}}
\def\ssosusp{\hbox{\scriptsize\bf s\hskip .1em}}
\def\gl#1{
{
\unitlength=.5pt
\begin{picture}(60.00,0.00)(0.00,0.00)
\thicklines
\put(2.00,-32.00){\makebox(0.00,0.00)[b]{\scriptsize $#1$}}
\put(0.00,0.00){\line(-1,-2){17.00}}
\put(0.00,0.00){\line(1,-2){17.00}}
\put(-18,-35.00){\line(1,0){35.00}}
\put(-15,-35.00){\line(0,-1){8.00}}
\put(-11,-35.00){\line(0,-1){8.00}}
\put(15,-35.00){\line(0,-1){8.00}}
\put(3,-37.00){\makebox(0.00,0.00)[t]{\scriptsize $\cdots$}}
\end{picture}}
}
\def\edge{{
\unitlength=1em
\begin{picture}(1.2,1)
\put(0.00,0.35){\line(1,0){1.00}}
\end{picture}
}}

\def\glvetsi#1{
{
\unitlength=.6pt
\begin{picture}(60.00,0.00)(0.00,0.00)
\thicklines
\put(2.00,-32.00){\makebox(0.00,0.00)[b]{\scriptsize $#1$}}
\put(0.00,0.00){\line(-1,-2){17.00}}
\put(0.00,0.00){\line(1,-2){17.00}}
\put(-18,-35.00){\line(1,0){35.00}}
\put(-15,-35.00){\line(0,-1){8.00}}
\put(-11,-35.00){\line(0,-1){8.00}}
\put(15,-35.00){\line(0,-1){8.00}}
\put(3,-37.00){\makebox(0.00,0.00)[t]{\scriptsize $\cdots$}}
\end{picture}}
}

\def\cases#1#2#3#4{
                  \left\{
                         \begin{array}{ll}
                           #1,\ &\mbox{#2}
                           \\
                           #3,\ &\mbox{#4}
                          \end{array}
                   \right.
}

\title[$n$-ary algebras]{(Non-)Koszulness of operads for 
       $\mbox{\large $n$}$-ary algebras, galgalim and other curiosities}

\author[Markl - Remm]{Martin Markl and Elisabeth Remm}
\thanks{The first author was supported by the grant GA \v CR
                  201/08/0397 and by
   the Academy of Sciences of the Czech Republic,
   Institutional Research Plan No.~AV0Z10190503.}
\address{Mathematical Institute of the Academy, {\v Z}itn{\'a} 25,
         115 67 Prague 1, The Czech Republic}
\email{markl@math.cas.cz}

\address{Laboratoire de Math\'ematiques et Applications,
        Universit\'e de Haute Alsace, Facult\'e des Sciences et
        Techniques, 4, rue des Fr\`eres Lumi\`ere,
        68093~Mulhouse~cedex, France.}
\email{Elisabeth.Remm@uha.fr}

\begin{document}

\begin{abstract}
We investigate operads for various $n$-ary algebras.  
As a useful tool we introduce galgalim --
analogs of the Lie-hedra for $n$-ary algebras. We then
focus to algebras with one anti-associative
operation.  We describe the relevant part of the deformation
cohomology for this type of algebras using the minimal model for the
anti-associative operad.  We also discuss free partially associative
algebras and formulate some open problems.
\end{abstract}

\bibliographystyle{plain}

\maketitle

\tableofcontents

\baselineskip17pt plus 1pt minus 1pt
\parskip3pt

\tableofcontents

\section*{Introduction}

We study Koszulness of operads for various $n$-ary algebras,
i.e.~algebras with an $n$-multi\-linear operation satisfying a specific
version of associativity.  In {\em
Section~\ref{Za_chvili_jdu_za_Jaruskou.}\/} we recall basic notions of
quadratic duality and Koszulness for quadratic operads and prove a
couple of related statements, emphasizing specific features of the
non-binary case which do not seem to have been addressed in
literature. Proposition~\ref{Jaruska_mi_udelala_svickovou!!}
describing the Poincar\'e series of generators of the minimal model
is, to our knowledge, a new one.

In {\em Section~\ref{Jaruska_je_moje_pusina.}\/} we introduce four
families of operads -- operads for totally resp.\ partially associative
$n$-algebras, and the operadic suspensions of these operads. In {\em
Section~\ref{Zitra-uz-budu-s-Jaruskou}\/} we define galgalim that,
in some sense, generalize the classical Stasheff's associahedra to the
realms of partially associative $n$-algebras. We will see that galgalim
encode some properties of free partially associative algebras.

In {\em Section~\ref{JarunKa}\/} we formulate and prove results
concerning Koszulness of operads for $n$-ary algebras. They are summed
up in the table of Figure~\ref{0}. 
We will then, in {\em Section \ref{sec:cohom-algebr-over}\/}, 
focus to the particular case of
algebras with one anti-associative operation, i.e.~an operation $a,b
\mapsto ab$ satisfying $a(bc) + a(bc) = 0$ for each $a$, $b$ and $c$.
The corresponding operad $\tildeAss$ is not Koszul, so the deformation
cohomology differs from the ``standard'' one. We describe 
the relevant part of the 
deformation cohomology based on the minimal model of $\tildeAss$. 

In {\em Section~\ref{Posledni_den_v_Mulhouse.}\/} we
give a description of the free partially associative algebras which, in
the Koszul cases, coincides with the one given
in~\cite{gnedbaye:CM96}. {\em Section~\ref{open}\/} formulates
open problems. 

Let us close this introduction by mentioning a couple of references
bearing some relation to the present article, namely 
the work of H.~Ataguema and
A.~Makhlouf~\cite{makhlouf-ataguema}, V.~Dotsenko and
A.~Khoroshkin~\cite{dotsenko-khoroshkin1},
A.V.~Gnedbaye~\cite{gnedbaye:CM96}, E.~Hoffbeck~\cite{hoffbeck} and
the talk given by J.-L.~Loday at the Winter School in Srn\'\i, in
January 2008.

\noindent 
{\bf Conventions.}  The basic reference for operads, quadratic duality and
Koszulness is~\cite{ginzburg-kapranov:DMJ94}, our notation and
terminology will also be based on~\cite{markl:handbook}
and~\cite{markl-shnider-stasheff:book}. We will work with operads in
the category of chain complexes over a field $\bfk$ of characteristic
zero though, in the light of~\cite{fresse:CM04}, most if not all results
remain valid over the ring of integers.

\section{Duality for quadratic operads revisited}
\label{Za_chvili_jdu_za_Jaruskou.}

Most of the ideas recalled in this section are implicitly present
in~\cite{getzler-jones:preprint,ginzburg-kapranov:DMJ94}, 
but we want to emphasize some specific
features of the non-binary case which do not seem to have been
addressed in literature. 

Fix a natural $n \geq 2$ and assume $E = \{E(a)\}_{a \geq 2}$ is a
$\Sigma$-module such that $E(a) = 0$ if $a \not=n$ and that, moreover,
$E(n)$ is finite-dimensional. 
We will study operads $\calP$ of the form
$\calP = \Gamma(E)/(R)$, where $\Gamma(E)$ is the free operad
generated by $E$ and $(R)$ the operadic ideal generated by a subspace
$R \subset \Gamma(E)(2n-1)$.
Operads of this type are called {\em quadratic\/}, or
{\em binary quadratic\/} if $n=2$.\footnote{Let us mention that, in
the original paper \cite{ginzburg-kapranov:DMJ94}, {\em quadratic\/}
always means {\em binary quadratic\/} in the terminology 
of the present note.}    
Let $E^\vee = \{E^\vee(a)\}_{a \geq 2}$ be a $\Sigma$-module with
\[
E^\vee(a) := \cases{\mbox{sgn}_a \otimes \uparrow^{a-2}E(a)^\#}
                   {if $a = n$ and}0{otherwise}
\]
where $\uparrow^{a-2}$ denotes the suspension
iterated $a-2$ times, $\mbox{sgn}_a$ the signum representation of the
symmetric group $\Sigma_a$, and $\#$ the linear dual of a graded
vector space with the induced representation.
Recall that $V^\#:=\Hom(V,\bfk)$, so $(V^\#)_d=(V_{-d})^\#$.
There is a non-degenerate, $\Sigma_{2n-1}$-equivariant pairing
\begin{equation}
\label{1}
\langle - | - \rangle :
\Gamma (E^\vee)(2n-1) \otimes \Gamma (E)(2n-1) \to \mbox {sgn}_{2n-1} 
\end{equation}
determined by requiring that
\[
\langle \uparrow^{n-2} e' \circ_i  \uparrow^{n-2} f'\ |\ 
 e'' \circ_j  f'' \rangle :=
\delta_{ij} 
(-1)^{(i+1)(n+1)} e'(e'') f'(f'') \in \bfk \cong \mbox {sgn}_{2n-1},
\]
for arbitrary $e'  ,f'\in E(n)^\#$, $e''  ,f'' \in E(n)$.

\begin{definition}
\label{v_patek_prijede}
The {\em Koszul\/} or {\em quadratic\/} dual of a quadratic operad
$\calP = \Gamma(E)/(R)$ as above is the quotient
\[
\calP^!  := \Gamma (E^\vee)/(R^\perp),
\] 
where $R^\perp \subset \Gamma
(E^\vee)(2n-1)$ is the annihilator of $R \subset
\Gamma(E)(2n-1)$ in the pairing~(\ref{1}), and $(R^\perp)$ the
operadic ideal generated by $R^\perp$.
\end{definition}

\begin{remark}
{\rm
\label{Last_day_in_Mulhouse}
If $\calP$ is a quadratic operad generated by
an operation of arity $n$ and degree $d$, then the generating operation
of $\calP^!$ has the same arity but
degree $-d + n-2$, i.e.~for $n\not =2$ (the non-binary case) the Koszul
duality {\em may not\/} preserve the degree of the generating
operation. As in the binary case, one easily verifies that
the quadratic dual is a
contravariant reflection, $(\calP^!)^! \cong \calP$.
}
\end{remark}

Recall that the
{\em operadic suspension\/} $\osusp\! E$ of a $\Sigma$-module $E\! =\!
\{E(a)\}_{a \geq 1}$ is the $\Sigma$-module $\osusp\! E = \{\osusp\!
E(a)\}_{a \geq 1}$, where $\osusp\! E(a) := \mbox{sgn}_a \otimes
\hskip -.4em \uparrow^{a-1} E(a)$.  It is a standard fact that, for a
dg-operad $\calP = \{\calP(a)\}_{a \geq 1}$, the operadic suspension
$\osusp\calP = \{\osusp\calP(a)\}_{a \geq 1}$ of the underlying
$\Sigma$-module is has a natural dg-operad structure. The operadic suspension
therefore extends from $\Sigma$-modules to an endofunctor on the category of
dg-operads. Likewise, the operadic suspension $\osusp \calC$ of a
dg-cooperad $\calC$ is a dg-cooperad.  We denote by
$\osusp^{-1}$ the inverse operation and call it, if necessary, the
operadic {\em desuspension\/}. 
In the following proposition, $\calP^\#$ denotes the
componentwise linear dual of a dg-operad with components of finite
type, with the obvious cooperad structure.
 
\begin{proposition}
\label{osud}
The free operad functor commutes with the operadic suspension, 
$\osusp \Gamma = \Gamma \osusp$. For a dg-operad $\calP$ with 
components of finite type, one has a natural isomorphism 
\[
(\osusp\calP)^\# \cong \osusp^{-1}( \calP^\#)
\] 
of dg-cooperads. Finally, if $\calP$ is a quadratic operad as in
Definition~\ref{v_patek_prijede}, its operadic suspension
$\osusp \calP$ is again quadratic and one has a natural isomorphism of
quadratic operads
\begin{equation}
\label{eq:1}
 (\osusp \calP)^! \cong  \osusp^{-1} (\calP^!).
\end{equation}
\end{proposition}

\begin{proof}
The first, rather nontrivial, claim of the proposition is the content
of~\cite[Proposition~II.3.20]{markl-shnider-stasheff:book}. The second
claim is obvious and the third can be verified directly.
\end{proof}

The {\em cobar
construction\/}~\cite[Definition~II.3.9]{markl-shnider-stasheff:book}
of a coaugmented cooperad $\calC$ is a dg-operad $\Omega(\calC)$ of
the form $\Omega(\calC) = (\Gamma(\downarrow\hskip -.5em\overline{ \osusp
\calC}),\pa_\Omega)$.  Here {\bf s} denotes the cooperadic suspension
recalled above, $\overline{\osusp \calC}$ the coaugmentation coideal
of the coaugmented cooperad $\osusp \calC$, and $\downarrow$ the
component-wise desuspension. The differential $\pa_\Omega$ is induced
by the structure operations of the cooperad $\calC$. If $\calP =
\{\calP(a)\}_{a \geq 1}$ is an augmented operad with
finite-dimensional components, the component-wise linear dual
$\calP^\# = \{\calP(a)^\#\}_{a \geq 1}$ is a coaugmented cooperad. 
The composition
$\Dbar(\calP) := \Omega(\calP^\#)$ of the linear dual with the cobar
construction is the {\em dual operad\/}
of~\cite[(3.2.12)]{ginzburg-kapranov:DMJ94}. In section II.3.3 of the
monograph~\cite{markl-shnider-stasheff:book}, $\Dbar(-)$ was called
the {\em dual bar construction\/}. We will use the latter terminology.

For $\calP$ quadratic, there clearly exist a natural map $\Dbar(\calP^!) \to
\calP$ of dg-operads.  The following definition is a straightforward
extension of~\cite[Definition~4.1.3]{ginzburg-kapranov:DMJ94}, allowing
that the quadratic operad $\calP$ need not be binary (i.e.~generated
by operations of arity two).

\begin{definition}
\label{Zitra_budu_s_Jaruskou!}
A quadratic operad $\calP$ is {\em Koszul\/} if the natural map
$\Dbar(\calP) \to \calP^!$ is a homology isomorphism.
\end{definition}

Let us close this section by formulating a couple of properties of
quadratic operads.

\begin{proposition}
\label{Jaruska_sbira_houbicky}
A quadratic operad as in
Definition~\ref{v_patek_prijede} is Kozsul if and only
if its operadic suspension $\osusp \calP$ is Koszul, i.e.~the operadic
suspension preserves Koszulness.
\end{proposition}

\begin{proof}
Assume that $\calP$ is Koszul. This, by definition, means that the
map \hbox{$\rho : \Dbar(\calP) \to \calP^!$} is a homology
isomorphism.  Since the operadic desuspension obviously preserves
homology isomorphisms, the desuspension of $\rho$,
\begin{equation}
\label{Jarina}
\osusp^{-1} \rho : \osusp^{-1} \Dbar(\calP) \to \osusp^{-1} (\calP^!)
\end{equation}
is a homology isomorphism, too. Expanding the definition of the
dual bar construction, one readily sees that the properties of the
operadic (de)suspension stated in Proposition~\ref{osud}
imply that
\[
\osusp^{-1} \Dbar(\calP) \cong  \Dbar( \osusp \calP)
\]
Combining this isomorphisms with~(\ref{Jarina}) and~(\ref{eq:1}), we
obtain a homology isomorphism $\Dbar(\osusp\calP) \to (\osusp
\calP)^!$, which coincides with the canonical map for the quadratic
operad $\osusp \calP$. This shows that $\osusp\calP$ is Koszul. To
prove that the Koszulness of $\osusp P$ implies the Koszulness of
$\calP$, all one needs to do is to reverse the steps of the proof of
the above implication.
\end{proof}

Observe that quadratic operads $\calP$ as we introduced them at the
beginning of this section have the properties that $\calP(1) \cong \bfk$
and that $\calP(a)$ is finite-dimensional for each $a \geq 1$. This
means that they are {\em admissible\/} in the sense 
of~\cite[(3.1.5)]{ginzburg-kapranov:DMJ94}. Therefore, all the
properties of the dual bar construction $\Dbar(-)$ proved in
\cite{ginzburg-kapranov:DMJ94} apply to our case. Namely, the
contravariant functor $\Dbar(-)$
preserves homology isomorphisms
\cite[Theorem~3.2.7b]{ginzburg-kapranov:DMJ94} and the canonical map
$\Dbar\big(\Dbar(\calP)\big) \to \calP$ is a homology isomorphism.
We also have the following extension of 
\cite[Proposition~4.1.4a]{ginzburg-kapranov:DMJ94} to the non-binary case.

\begin{proposition}
\label{sec:dual-quadr-oper}
A quadratic operad $\calP$ is Koszul if and only if so is $\calP^!$,
i.e.~the quadratic duality preserves Koszulness.
\end{proposition}

\begin{proof}
A verbatim transcription of the corresponding statement of
\cite{ginzburg-kapranov:DMJ94}. Suppose that $\calP$ is Koszul and 
let $\rho : \Dbar(\calP) \to \calP^!$
be the canonical map. One then has the composition
\[
\Dbar(\calP^!) \stackrel{\Dbar(\rho)}{\longrightarrow}
\Dbar\big(\Dbar(\calP)\big) \to \calP
\]
which is, due to the properties of the dual bar construction recalled
above, a homology isomorphism. It is immediate that this composition
coincides with the canonical map $\Dbar(\calP) \to \calP^!$ for
$\calP^!$. So the Koszulness of $\calP$ implies the Koszulness of
$\calP^!$. The opposite implication is obtained by applying the above
arguments to $\calP^!$ instead of $\calP$. 
\end{proof}

The {\em Poincar\'e\/} or {\em generating series\/} of a graded operad
$\calP_* = \{\calP_*(a)\}_{a \geq 1}$ with finite-dimensional
components is defined by
\[
g_\calP(t) := \sum_{a \geq 1} \frac 1{a!} \chi(\calP(a)) t^a,
\]
where $\chi(\calP(a))$ denotes the Euler characteristic of the graded
vector space $\calP_*(n)$,
\[
\chi(\calP(a)) := \sum_i (-1)^i \dim(\calP_i(a)).
\]

Recall that each operad $\calP$ with $\calP(1) = \bfk$ admits a {\em
minimal model\/}, unique up to isomorphism
\cite[II.3.10]{markl-shnider-stasheff:book}. This is, by definition, a
homology isomorphism $(\calP,0) \stackrel \rho\leftarrow
(\Gamma(M),\pa)$ from the free operad $\Gamma(M)$ on a collection $M =
\{M(a)\}_{a \geq 2}$, equipped with a differential $\pa$, to the
operad $\calP$ with the trivial differential. The {\em minimality\/}
requires that $\pa(M)$ consists of decomposable elements of the free
operad $\Gamma(M)$.  The following proposition relates the generating
series of $\calP$ and the generating series of the collection of generators
of its minimal model.

\begin{proposition}
\label{Jaruska_mi_udelala_svickovou!!}
Let $\calP$ be an arbitrary operad with $\calP(1) = \bfk$ and
finite-dimensional pieces. Let $(\calP,0) \stackrel \rho\leftarrow
(\Gamma(M),\pa)$ be its minimal model. The Poincar\'e series $g_\calP(t)$
of $\calP$ is related with the generating function
\[
g_M(t) := -t + \sum_{a \geq 2} \frac 1{a!} \chi(M(a))t^a
\] 
of the $\Sigma$-module $M=\{M(a)\}_{a \geq 2}$ by the functional
equation
\begin{equation}
\label{Jaruska_mi_udelala_svickovou!}
g_\calP(-g_M(t)) = t.
\end{equation}
In other words,~$g_M(t)$ is the formal inverse of $g_\calP(t)$ taken with the
opposite sign. 
\end{proposition}

\begin{proof}
The statement can be verified by repeating the steps of the proof
of~\cite[Theorem~3.3.2]{ginzburg-kapranov:DMJ94}. Since our
Proposition \ref{Jaruska_mi_udelala_svickovou!!}  does not seem to be
commonly known, we decided to prove it here, not just refer
to~\cite{ginzburg-kapranov:DMJ94}.

Let us convert first~(\ref{Jaruska_mi_udelala_svickovou!}) into an
equivalent form, more suitable for the purposes of this proof. 
The substitution $t \mapsto g_\calP(t)$
brings~(\ref{Jaruska_mi_udelala_svickovou!}) into
\[
g_\calP(-g_M(g_\calP(t))) = g_\calP(t).
\] 
Applying $g^{-1}_\calP$ to both sides of the above equation leads to
\begin{equation}
\label{eq:2}
-g_M(g_\calP(t)) = t.
\end{equation}
Since $g_M$ is formally invertible, (\ref{eq:2}) is 
equivalent to~(\ref{Jaruska_mi_udelala_svickovou!}).

Recall that the free operad $\Gamma(M)$ is spanned by rooted trees
with vertices decorated by elements of the generating collection $M$,
see \cite[II.1.9]{markl-shnider-stasheff:book} 
for the precise meaning of this statement.
It follows from this observation that, for each $a \geq 2$,
its arity $a$ piece $\Gamma(M)(a)$ decomposes as
\[
\Gamma(M)(a) = \bigoplus_{r \geq 2}\bigoplus_{S_r(a)}
\Gamma(M)(\Rada u1r), 
\]
where
\[
S_r(a) := \big\{(\Rada u1r) \in {\mathbb Z}^r;\
\Rada u1r \geq 1,\ u_1 + \cdots + u_r = a\big\} 
\]
and $\Gamma(M)(\Rada u1r) \subset \Gamma(M)(a)$ is the subspace
spanned by elements of the form
\begin{equation}
\label{eq:3}
(\cdots ((m \circ_r x_r) \circ_{r-1} x_{r-1} ) \cdots) \circ_1 x_1
\end{equation}
with some $m \in M(r) \subset \Gamma(M)(r)$ and 
$x_i \in \Gamma(M)(u_i)$, $1\leq i \leq r$.
In terms of trees, the expressions in~(\ref{eq:3}) can be depicted as
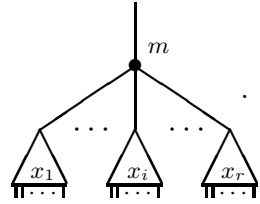
\[
\raisebox{-45pt}{\rule{0pt}{0pt}}
\unitlength=1.2pt
\begin{picture}(60.00,30.00)(0.00,40.00)
\thicklines
\put(17.00,30.00){\makebox(0.00,0.00){$\cdots$}}
\put(47.00,30.00){\makebox(0.00,0.00){$\cdots$}}
\put(30.00,50.00){\makebox(0.00,0.00){$\bullet$}}
\put(34.00,54.00){\makebox(0.00,0.00)[lb]{\scriptsize $m$}}
\put(30.00,50.00){\line(0,-1){20.00}}
\put(30.00,50.00){\line(3,-2){30.00}}
\put(30.00,50.00){\line(-3,-2){30.00}}
\put(30.00,50.00){\line(0,1){20.00}}
\put(15.00,30.00){\makebox(0.00,0.00)[t]{\glvetsi{x_1}}}
\put(45.00,30.00){\makebox(0.00,0.00)[t]{\glvetsi{x_i}}}
\put(75.00,30.00){\makebox(0.00,0.00)[t]{\glvetsi{x_r}}}
\end{picture}\ .
\]
Simple representation theory and combinatorics implies that
\[
\chi\big(\Gamma(M)(\Rada u1r)\big) = \frac{a!}{r!\cdot u_1! \cdots u_r!}\
\chi(M(r)) \cdot
\chi(\Gamma(M)(u_1)) \cdots \chi( \Gamma(M)(u_r))
\]
therefore
\[
\chi\big(\Gamma(M)(a)\big) =  \sum_{r \geq 2}\sum_{S_r(a)} 
\frac{a!}{r!\cdot u_1! \cdots u_r!}\
\chi(M(r)) \cdot
\chi(\Gamma(M)(u_1)) \cdots \chi( \Gamma(M)(u_r)).
\]
Since the minimal model map $\rho$ is a 
homology isomorphism, $\chi(\Gamma(M)(a)) =
\chi(\calP(a))$ for each $a \geq 1$ and the above display implies
\begin{equation}
\label{eq:4}
\chi(\calP(a)) =  \sum_{r \geq 2}\sum_{S_r(a)} 
\frac{a!}{r!\cdot u_1! \cdots u_r!}\
\chi(M(r)) \cdot
\chi(\calP(u_1)) \cdots \chi( \calP(u_r))
\end{equation}
for each $a \geq 2$.
To make the following main argument of the proof more transparent, we denote
\[
\alpha_a := \chi(\calP(a)) \mbox { for $a \geq 1$ and } 
\beta_a := \chi(M(a)) \mbox { for $a \geq 2$}
\]
so that 
\[
g_\calP = \sum_{a \geq 1} \alpha_a t^a
\mbox { and }
g_M =  -t+ \sum_{a \geq 2} \beta_a t^a.
\]
Then~(\ref{eq:4}) reads
\[
\alpha_a =  \sum_{r \geq 2}\sum_{S_r(a)} 
\frac{a!}{r!\cdot u_1! \cdots u_r!}\  \beta_r \cdot
\alpha_{u_1} \cdots \alpha_{u_r},
\]
for $a \geq 2$.
Elementary calculus shows that the above equation is precisely the
recursion that ties the coefficients of the power series 
$g_\calP$ and $g_M$ satisfying~(\ref{eq:2}), hence
also~(\ref{Jaruska_mi_udelala_svickovou!}).  
\end{proof}

The following important criterion of Koszulness, which is a verbatim
generalization of \cite[Theorem~3.3.2]{ginzburg-kapranov:DMJ94},
follows easily from Proposition~\ref{Jaruska_mi_udelala_svickovou!!}.

\begin{theorem}
\label{gk}
If a quadratic operad $\mathcal{P}$ is Koszul, then its Poincar\'e series and
the Poincar\'e series of its dual $\calP^!$ are tied by
the functional equation
\begin{equation}
\label{zitra_Jarka}
g_{\mathcal{P}}\big(-g_{\mathcal{P}^!}(-t)\big)=t.
\end{equation}
\end{theorem}

\begin{proof}
If $\calP$ is quadratic Koszul, then
its minimal model is isomorphic to the dual bar construction 
$\Dbar(\calP^!)$ of its Koszul dual $\calP^!$. The dual bar
construction is, as a graded operad, generated by the
$\Sigma$-collection $\downarrow \hskip -.1em  \overline{\rule{0pt}{.9em}{\bf s}
  \calP^!}^\# = \{\uparrow^{a-2}
\hskip -.1em  \calP^!(a)^\#\}_{a \geq 2}$. So, in the
Koszul case
\[
g_M(t) = g_{\downarrow   \overline{{\bf s}
  \calP^!}^\#}(t) = g_{\calP^!}(-t), 
\]  
which, substituted to~(\ref{Jaruska_mi_udelala_svickovou!}), 
gives ~(\ref{zitra_Jarka}).
\end{proof}

Let us close this section with another criterion for Koszulness. 
Denote by $\Gamma^2(M)$ the subcollection of $\Gamma(M)$ spanned by
expressions with precisely two instances of elements of the generating
collection $M$ or, equivalently, by $M$-decorated trees with
two vertices.  We say that the minimal model
$(\Gamma(M),\pa)$ of $\calP$ is {\em quadratic\/} if $\pa(M) \subset
\Gamma^2(M)$. 

\begin{fact}
\label{sec:dual-quadr-oper-1}
A quadratic Koszul operad has a quadratic minimal model.
\end{fact}

Indeed, if $\calP$ is Koszul, by Proposition~\ref{sec:dual-quadr-oper}
so is $\calP^!$. This, by definition, means that the natural map
$\Dbar(\calP^!) \to \calP$ is a homology isomorphism, therefore it is
a quadratic minimal model of $\calP$.

We are aware that Fact~\ref{sec:dual-quadr-oper-1} is a very
simple-minded Koszulness test. Yet, we will see in
Section~\ref{sec:cohom-algebr-over} that the non-Koszulness of the
operad $\tildeAss$ for anti-associative algebras can be proved by
showing that it does not admit a quadratic minimal model.  It is also
possible that the non-Koszulness of the operads $\tAss^n_d$ introduced
in the following section can be, for $n \geq 8$ and $d$ odd,
established using Fact~\ref{sec:dual-quadr-oper-1}, while the
Ginzburg-Kapranov test (Theorem~\ref{gk}) may not be
determinative. See also a discussion in~\cite{n-alg}.

\section{Four families of $n$-ary algebras} 
\label{Jaruska_je_moje_pusina.}

We introduce four families of quadratic operads and describe their
Koszul duals. These families cover most of examples of `$n$-ary
algebras' with one operation without symmetry which we were able to find in the
literature.  

Let $V$ be a graded vector space, $n \geq 2$, and $\mu : \otexp Vn \to V$ 
a degree $d$ multilinear operation symbolized by
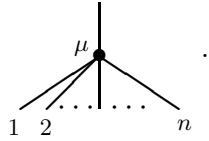
\begin{center}
\unitlength=1.000000pt
\begin{picture}(60.00,40.00)(0.00,30.00)
\thicklines
\put(22.00,30.00){\makebox(0.00,0.00){$\cdots$}}
\put(42.00,30.00){\makebox(0.00,0.00){$\cdots$}}
\put(30.00,50.00){\makebox(0.00,0.00){$\bullet$}}
\put(26.00,50.00){\makebox(0.00,0.00)[rb]{\scriptsize $\mu$}}
\put(70.00,50.00){\makebox(0.00,0.00){.}}
\put(30.00,50.00){\line(0,-1){20.00}}
\put(30.00,50.00){\line(3,-2){30.00}}
\put(30.00,50.00){\line(-1,-1){20.00}}
\put(30.00,50.00){\line(-3,-2){30.00}}
\put(30.00,50.00){\line(0,1){20.00}}
\put(0.00,26.00){\makebox(0.00,0.00)[rt]{\scriptsize $1$}}
\put(12.00,26.00){\makebox(0.00,0.00)[rt]{\scriptsize $2$}}
\put(65.00,26.00){\makebox(0.00,0.00)[rt]{\scriptsize $n$}}
\end{picture}
\end{center}
We say that $A = (V,\mu)$ is a {\em degree $d$ totally associative
$n$-ary algebra\/} if, for each $1 \leq i,j \leq n$,  
\[
\mu\left(\id^{\otimes i-1} \otimes \mu \otimes \id^{\otimes
  n-i}\right) 
= \mu\left(\id^{\otimes j-1} \otimes \mu \otimes \id^{\otimes
  n-j}\right),
\] 
where $\id : V \to V$ is the identity map. Graphically, we demand that
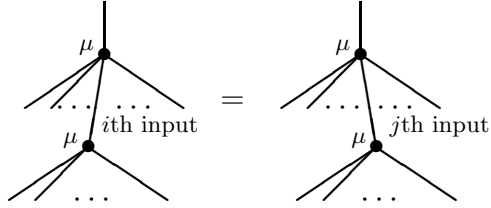
\begin{center}
\unitlength=1.000000pt
\begin{picture}(60.00,70.00)(0.00,0.00)
\thicklines
\put(42.00,30.00){\makebox(0.00,0.00){$\cdots$}}
\put(30.00,50.00){\makebox(0.00,0.00){$\bullet$}}
\put(26.00,50.00){\makebox(0.00,0.00)[rb]{\scriptsize $\mu$}}
\put(30.00,50.00){\line(3,-2){30.00}}
\put(30.00,50.00){\line(-1,-1){20.00}}
\put(30.00,50.00){\line(-3,-2){30.00}}
\put(30.00,50.00){\line(0,1){20.00}}
\put(19.5,30.00){\makebox(0.00,0.00){$\cdots$}}
\put(29.00,23.00){\makebox(0.00,0.00)[l]{\scriptsize $i$th input}}
\put(30.00,50.00){\line(-1,-6){5.70}}
\put(-6,0){
\put(32.00,-5.00){\makebox(0.00,0.00){$\cdots$}}
\put(30.00,15.00){\makebox(0.00,0.00){$\bullet$}}
\put(26.00,15.00){\makebox(0.00,0.00)[rb]{\scriptsize $\mu$}}
\put(30.00,15.00){\line(3,-2){30.00}}
\put(30.00,15.00){\line(-1,-1){20.00}}
\put(30.00,15.00){\line(-3,-2){30.00}}
}
\end{picture}
\hskip 10pt \raisebox{30pt}{=} \hskip 10pt
\unitlength=1.000000pt
\begin{picture}(60.00,70.00)(0.00,0.00)
\thicklines
\put(44.00,30.00){\makebox(0.00,0.00){$\cdots$}}
\put(30.00,50.00){\makebox(0.00,0.00){$\bullet$}}
\put(26.00,50.00){\makebox(0.00,0.00)[rb]{\scriptsize $\mu$}}
\put(30.00,50.00){\line(3,-2){30.00}}
\put(30.00,50.00){\line(-1,-1){20.00}}
\put(30.00,50.00){\line(-3,-2){30.00}}
\put(30.00,50.00){\line(0,1){20.00}}
\put(24.00,30.00){\makebox(0.00,0.00){$\cdots$}}
\put(41.00,23.00){\makebox(0.00,0.00)[l]{\scriptsize $j$th input}}
\put(30.00,50.00){\line(1,-6){5.70}}
\put(6,0){
\put(32.00,-5.00){\makebox(0.00,0.00){$\cdots$}}
\put(30.00,15.00){\makebox(0.00,0.00){$\bullet$}}
\put(26.00,15.00){\makebox(0.00,0.00)[rb]{\scriptsize $\mu$}}
\put(30.00,15.00){\line(3,-2){30.00}}
\put(30.00,15.00){\line(-1,-1){20.00}}
\put(30.00,15.00){\line(-3,-2){30.00}}
}
\end{picture}
\end{center}
for each $i$, $j$ for which the above compositions make sense. Observe
that degree $0$ totally associative $2$-algebras are ordinary
associative algebras. 

In the following definitions, $\Gamma(\mu)$ will
denote the free operad on the $\Sigma$-module $E_\mu$ with
\[
E_\mu(a) = \cases{\mbox{the regular representation }  \bfk[\Sigma_n]
  \mbox{ generated by } \mu}{if $a = n$ and}0{otherwise.}
\]

\begin{definition}
We denote $\tAss^n_d$ the operad for totally associative $n$-ary
algebras with operation in degree $d$, that is,
\[
\tAss^n_d : =\Gamma (\mu)/(R_{\sstAss^n_d} )
\] 
with $\mu$ an arity $n$ generator of degree $d$ and
\[
R_{\sstAss^n_d} : =\Span \left\{\mu \circ_i\mu- \mu \circ_j\mu, 
\ {\rm for} \ i,j=1,\ldots ,n \right\}.
\]
\end{definition}

We call $A = (V,\mu)$ a {\em degree $d$ partially associative
$n$-ary algebra\/} if the following single axiom is satisfied:  
\begin{equation}
\label{Jaruska_sibalsky_mrka.}
\sum_{i=1}^n (-1)^{(i+1)(n-1)}
\mu\left(\id^{\otimes i-1} \otimes \mu \otimes \id^{\otimes n-i}\right) =0.
\end{equation}  

Degree $0$ partially associative $2$-ary algebras are classical
associative algebras. More interesting observation is that degree
$(n-2)$ partially associative $n$-ary algebras are the same as
$A_\infty$-algebras $A = (V,\mu_1,\mu_2,\ldots)$~\cite[\S1.4]{markl:JPAA92} 
which are ``meager'' in that they satisfy $\mu_k = 0$ for $k \not=n$.
Symmetrizations of these {\em meager $A_\infty$-algebras\/} are 
{\em Lie $n$-algebras\/} in the sense of~\cite{hanlon-wachs:AdvMa:95}. 

\begin{definition}
We denote $\pAss^n_d$ the operad for partially associative $n$-ary
algebras with operation in degree $d$. Explicitly,
\[
\pAss^n_d:=
\Gamma (\mu)/\left(\rule{0pt}{1em}\right. 
\sum_{i=1}^n (-1)^{(i+1)(n-1)}\mu \circ_i
\mu \left. \rule{0pt}{1em} \right)
\]
with $\mu$ a generator of degree $d$ and arity $n$.
\end{definition}

It follows from the above remarks that $\tAss^2_0= \pAss^2_0=\Ass$,
where $\Ass$ is the operad for associative algebras.
We are going to introduce the remaining two families of
operads. Recall that $\osusp$ denotes the operadic suspension and
$\osusp^{-1}$ the obvious inverse operation.

\begin{definition}
We define $\ttildeAss^n_d:= \osusp \tAss^n_{d-n+1}$ and
$\ptildeAss^n_d:= \osusp^{-1} \pAss^n_{d+n-1}$.
\end{definition}

We leave as an exercise to verify that $\ttildeAss^n_d$-algebras are
structures $A = (V,\mu)$, where $\mu : \otexp Vn \to V$ is a degree
$d$ linear map satisfying,
for each $1 \leq i,j \leq n$,  
\[
(-1)^{i(n+1)}
\mu\left(\id^{\otimes i-1} \otimes \mu \otimes \id^{\otimes
  n-i}\right) 
= (-1)^{j(n+1)}\mu\left(\id^{\otimes j-1} \otimes \mu \otimes \id^{\otimes
  n-j}\right).
\] 
Likewise,  $\ptildeAss^n_d$-algebras are similar structures, but this
time satisfying
\[
\sum_{i=1}^n \mu\left(\id^{\otimes i-1} 
\otimes \mu \otimes \id^{\otimes n-i}\right) =0.
\]  

\begin{definition}
\label{Jar}
Let $\tildeAss := \ttildeAss^2_0 = \ptildeAss^2_0$. Explicitly,
$\tildeAss$-algebras are structures $A = (V,\mu)$ with a degree $0$
bilinear operation $\mu : V \ot V \to V$ satisfying
\[
\mu(\mu \ot \id) + \mu(\id \ot \mu) =0
\]
or, in elements
\begin{equation}
\label{Jaruska_je_pusinka}
a(bc) + (ab)c = 0,
\end{equation}
for $a,b,c \in V$. We call these objects {\em anti-associative\/} algebras.
\end{definition}

Anti-associative algebras can be viewed as associative algebras with
the associativity taken with the opposite sign which explains their
name. Similarly, $\tAss^2_1=\pAss^2_1$-algebras are associative
algebras with operation of degree~$1$. The corresponding, essentially
equivalent, operads are the simplest examples of non-Koszul operads,
as we will see in Section~\ref{JarunKa}. The proof of the following
proposition is an exercise.

\begin{proposition}
\label{JarKA}
For each $n \geq 2$ and $d$, $(\tAss^n_d)^! = \pAss^n_{-d +n-2}$,
$(\pAss^n_d)^! = \tAss^n_{-d +n-2}$,  $(\ttildeAss^n_d)^! = \ptildeAss^n_{-d +n-2}$ and  
$(\ptildeAss^n_d)^! = \ttildeAss^n_{-d +n-2}$, 
\end{proposition}

\section{Sundry facts about $n$-ary algebras}
\label{Zitra-uz-budu-s-Jaruskou}

In this section we discuss two constructions (galgalim and higher
associahedra) that, in some sense, generalize classical Stasheff's
associahedra to the realms of partially resp.\ totally associative
$n$-algebras. We also show how galgalim encode some properties of free
partially associative algebras. Necessary facts about the associahedra
can be gained from \cite[II.1.6]{markl-shnider-stasheff:book} or from the
original source~\cite{stasheff:TAMS63}.

\vskip .3em
\noindent 
{\bf Galgalim.}  This part is devoted to degree $0$ partially
associative $n$-algebras, i.e.~to algebras over the operad
$\pAss^n_0$. The fact that, for $n \geq 3$, their defining
axiom~(\ref{Jaruska_sibalsky_mrka.})  has more than two terms rules
out the existence of an analog of the Stasheff associahedra -- the
edges of such a hypothetic polyhedra ought to have more than two
end-points. One can, however, still draw some graphs that visualize
the relations among the axioms, similar to the Lie-hedron constructed
in~\cite{markl-shnider:coherence}.  Their nature is somehow dual to
the nature of the associahedra in that their vertices are indexed by
the defining {\em relations\/}, while their edges are labelled by the
iterated structure {\em operations\/}.

Let us start with the case $n=2$, when $\pAss^n_0$ is the operad for
associative algebras, so the
associahedra actually exist. 
There are five ways to apply a binary operation to four elements:
\begin{equation}
\label{eq:5}
((\bullet\bullet)\bullet)\bullet,\
(\bullet(\bullet\bullet))\bullet,\
(\bullet\bullet) (\bullet\bullet),\
\bullet((\bullet\bullet)\bullet),\
\bullet(\bullet(\bullet\bullet)).
\end{equation}
There are five relations between these expressions obtained by one
instance of the axiom~(\ref{Jaruska_sibalsky_mrka.}) which is, for $n=2$,
the associativity, namely
\begin{align}
\nonumber 
(\bullet\bullet) (\bullet\bullet)& -\bullet(\bullet(\bullet\bullet)) =0
  \mbox { which we denote } \bullet\!\bullet(\bullet\bullet),
\\ \nonumber 
\bullet(\bullet(\bullet\bullet))&- \bullet((\bullet\bullet)\bullet) =0
 \mbox { which we denote } \bullet\!(\tri),
\\
\label{eq:6}
\bullet((\bullet\bullet)\bullet) &- (\bullet(\bullet\bullet))\bullet =0
\mbox { which we denote } \bullet\!(\bullet\bullet)\bullet,
\\ \nonumber 
(\bullet(\bullet\bullet))\bullet &- ((\bullet\bullet)\bullet)\bullet =0
\mbox { which we denote } (\tri)\bullet, \mbox { and}
\\ \nonumber 
((\bullet\bullet)\bullet)\bullet &- (\bullet\bullet) (\bullet\bullet) =0
\mbox { which we denote } (\bullet\bullet)\!\bullet\!\bullet.
\end{align}

We call these relations {\em elementary\/}.  Observe that each symbol
listed in~(\ref{eq:5}) appears in precisely two elementary relations
of~(\ref{eq:6}).  So we may draw a graph with edges labelled by the
five symbols in~(\ref{eq:5}) which share a common vertex if and only
if they labels appear in the same relation of~(\ref{eq:6}). The common
vertex emerging in this way will be labelled by this relation. We get
a graph with five vertices and five edges:

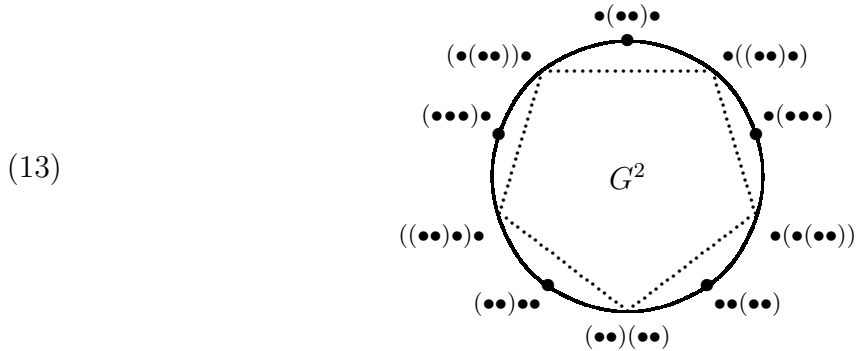
\begin{equation}
\label{G2}
\raisebox{-2.3cm}{\rule{0pt}{0pt}}
\unitlength=.9cm
\begin{picture}(6,1.8)(-3.00,0)
\thicklines
\put(0,0){\hustykrouzek}    
\put(0,2){\makebox(0,0){$\bullet$}}
\put(0,2.2){\makebox(0,0)[b]{\scriptsize $\bullet(\b\b)\b$}}
\put(1.902,0.618){\makebox(0,0){$\bullet$}}
\put(2.002,0.718){\makebox(0,0)[lb]{\scriptsize  $\bullet(\tri)$}}
\put(-1.902,0.618){\makebox(0,0){$\bullet$}}
\put(-2.002,0.718){\makebox(0,0)[rb]{\scriptsize  $(\tri)\b$}}
\put(1.175,-1.618){\makebox(0,0){$\bullet$}}
\put(1.275,-1.718){\makebox(0,0)[lt]{\scriptsize$\dve(\dve)$}}
\put(-1.175,-1.618){\makebox(0,0){$\bullet$}}
\put(-1.275,-1.718){\makebox(0,0)[rt]{\scriptsize$(\dve)\dve$}}
\put(1.4,1.6){\makebox(0,0)[lb]{\scriptsize  $\b((\b\b)\b)$}}
\put(-1.4,1.6){\makebox(0,0)[rb]{\scriptsize  $(\b(\b\b))\b$}}
\put(2.1,-.7){\makebox(0,0)[lt]{\scriptsize  $\b(\b(\b\b))$}}
\put(-2.1,-.7){\makebox(0,0)[rt]{\scriptsize  $((\b\b)\b)\b$}}
\put(0,-2.2){\makebox(0,0)[t]{\scriptsize  $(\b\b)(\b\b)$}}
\thinlines
\multiput(-1.26,1.54)(0.12,0){22}{\makebox(0,0){$\cdot$}}
\multiput(-1.26,1.54)(-0.03,-.1){22}{\makebox(0,0){$\cdot$}}
\multiput(1.26,1.54)(0.03,-.1){22}{\makebox(0,0){$\cdot$}}
\multiput(0,-2)(0.085,.065){22}{\makebox(0,0){$\cdot$}}
\multiput(0,-2)(-0.085,.065){22}{\makebox(0,0){$\cdot$}}
\put(0,0){\makebox(0,0){$G^2$}}
\end{picture}
\end{equation}
which is dual to the $1$-skeleton of the Stasheff
pentagon $K_4$ indicated by the dotted lines.

For $n=3$, there are $12$ ways to multiply  $7$ elements by a ternary
operation:
\begin{equation}
\label{eq:7}
\begin{aligned}
((\tri)\dve)\dve,\ (\b(\dve)\b)\dve,\ (\dve(\tri))\dve,&\
\b((\tri)\dve)\b,\ \b(\b(\dve)\b)\b,\ \b(\dve(\tri))\b,
\\
\dve((\tri)\dve),\ \dve(\b(\dve)\b)\b,\ \dve(\dve(\tri)),&\
(\tri)(\tri)\b,\ (\tri)\b(\tri),\ \b(\tri)(\tri)
\end{aligned}
\end{equation}
and $8$ elementary relations between these terms obtained by one
instance of the partial associativity $(\tri)\dve + \b(\tri)\b +\,
\dve(\tri)$, namely
\begin{align*}
\R1 \mbox { denoting } &\ (\dve(\tri))\dve + \b(\b(\tri)\b)\!\b +
\dve((\tri)\dve) =0,
\\
\R2\mbox { denoting } &\ \dve((\tri)\dve) + \dve(\b(\tri)\b) +
\dve(\dve(\tri))  =0,
\\
\R3 \mbox { denoting }  &\  
\b\!(\tri)(\tri)\ + (\tri)\!\b\!(\tri) +  \dve(\dve(\tri)) =0.
\\
\R4\mbox { denoting } &\ (\b(\tri)\b)\dve + \b((\tri)\dve)\!\b
+ \b\!(\tri)(\tri) =0,
\\
\R5\mbox { denoting } &\ \b\!((\tri)\dve)\!\b + \b\!(\b(\tri)\b)\!\b +
\b\!(\dve(\tri))\b =0,
\\
\R6\mbox { denoting } &\ (\tri)(\tri)\!\b + \b\!(\dve(\tri))\!\b
+ \dve(\b(\tri)\b)  =0,
\\
\R7 \mbox { denoting }  &\ ((\tri)\dve)\dve + 
 (\tri)\!\b\!(\tri) + (\tri)(\tri) \b  =0,  \mbox { and}
\\
\R8\mbox { denoting } &\ ((\tri)\dve)\dve + (\b(\tri)\b)\dve +
(\dve(\tri))\dve =0,
\end{align*}
Each element of~(\ref{eq:7}) again appears in precisely $2$
elementary relations. The corresponding graph with $12$ edges indexed by
expressions~(\ref{eq:7}) and $8$ vertices labelled by elementary
relations is the wheel with eight spikes:
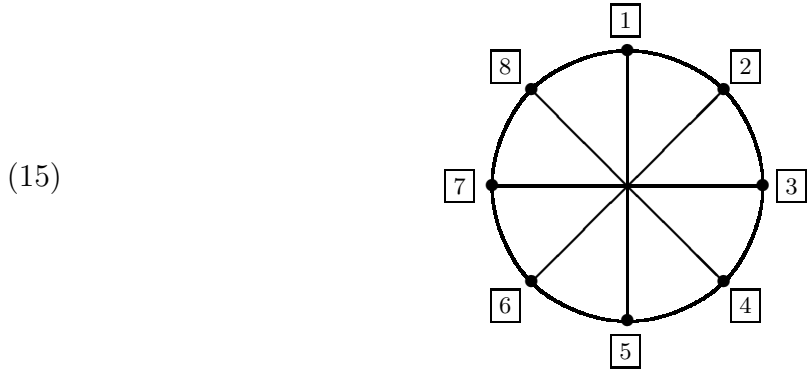
\begin{equation}
\label{G3}
\raisebox{-2.3cm}{\rule{0pt}{0pt}}
\unitlength=.9cm
\begin{picture}(6,2.6)(-3.00,0)
\thicklines
\put(0,0){\hustykrouzek}    
\put(0,2){\makebox(0,0){$\bullet$}}
\put(0,2.2){\makebox(0,0)[b]{\scriptsize $\R1$}}
\put(0,2){\line(0,-1){4}}
\put(-1.42,-1.42){\line(1,1){2.84}}
\put(-1.42,1.42){\line(1,-1){2.84}}
\put(-2,0){\line(1,0){4}}
\put(2,0){\makebox(0,0){$\bullet$}}
\put(2.2,0){\makebox(0,0)[l]{\scriptsize $\R3$}}
\put(0,-2){\makebox(0,0){$\bullet$}}
\put(0,-2.2){\makebox(0,0)[t]{\scriptsize $\R5$}}
\put(-2,0){\makebox(0,0){$\bullet$}}
\put(-2.2,0){\makebox(0,0)[r]{\scriptsize $\R7$}}
\put(1.42,1.42){\makebox(0,0){$\bullet$}}
\put(1.52,1.52){\makebox(0,0)[lb]{\scriptsize $\R2$}}
\put(-1.42,1.42){\makebox(0,0){$\bullet$}}
\put(-1.52,1.52){\makebox(0,0)[rb]{\scriptsize $\R8$}}
\put(1.42,-1.42){\makebox(0,0){$\bullet$}}
\put(1.52,-1.52){\makebox(0,0)[lt]{\scriptsize $\R4$}}
\put(-1.42,-1.42){\makebox(0,0){$\bullet$}}
\put(-1.52,-1.52){\makebox(0,0)[tr]{\scriptsize $\R6$}}
\end{picture}
\end{equation}

Observe that elementary relations have a left-right mirror
symmetry: $\R1$ and $\R5$ are self-symmetric, while the mirror image
of $\R2$ is $\R8$, the image of $\R3$ is $\R7$ and the image of $\R4$
is $\R6$. This symmetry is reflected by the left-right symmetry of (\ref{G3}).

For $n=4$, there are $22$ ways of applying a $4$-ary operation to $10$
elements, and $11$ elementary relations among these elements. 
The resulting graph is shown in Figure~\ref{G4}. 
The $5$th galgal (case $n=5$) has $14$ vertices and $35$ edges, its
portrait is given in Figure~\ref{G5}. We call these figures
{\em galgalim\/} (plural of {\em galgal\/}), the Hebrew for wheel.

\begin{figure}
\begin{center}
\unitlength=1.5cm
\begin{picture}(6,4)(-3.00,-2)
\thicklines
\put(0,0){\hustykrouzek}    
\put(0,0){\makebox(0,0){$G^4$}}    
\put(0,2){\makebox(0,0){$\bullet$}} 
\multiput(0,2)(0.009,-.0103){220}{\makebox(0,0){$\cdot$}} 
\put(1.08,1.682){\makebox(0,0){$\bullet$}} 
\multiput(1.08,1.682)(0.002,-.0133){225}{\makebox(0,0){$\cdot$}}
\put(1.818,0.830){\makebox(0,0){$\bullet$}} 
\multiput(1.818,0.830)(-0.0056,-.0122){225}{\makebox(0,0){$\cdot$}} 
\put(1.978,-0.284){\makebox(0,0){$\bullet$}} 
\multiput(1.978,-0.284)(-0.0117,-.00755){220}{\makebox(0,0){$\cdot$}} 
\put(1.5294,-1.309){\makebox(0,0){$\bullet$}} 
\multiput(1.5294,-1.309)(-0.014,0){220}{\makebox(0,0){$\cdot$}}
\put(.5634,-1.918){\makebox(0,0){$\bullet$}} 
\multiput(.5634,-1.918)(-0.0117,.00755){220}{\makebox(0,0){$\cdot$}}
\put(-1.08,1.682){\makebox(0,0){$\bullet$}} 
\multiput(-1.08,1.682)(0.013,-.0038){220}{\makebox(0,0){$\cdot$}} 
\put(-1.818,0.830){\makebox(0,0){$\bullet$}} 
\multiput(-1.818,0.830)(0.013,.0038){220}{\makebox(0,0){$\cdot$}} 
\put(-1.978,-0.284){\makebox(0,0){$\bullet$}} 
\multiput(-1.978,-0.284)(0.009,.0103){220}{\makebox(0,0){$\cdot$}} 
\put(-1.5294,-1.309){\makebox(0,0){$\bullet$}} 
\multiput(-1.5294,-1.309)(0.002,.0133){225}{\makebox(0,0){$\cdot$}}
\put(-.5634,-1.918){\makebox(0,0){$\bullet$}} 
\multiput(-.5634,-1.918)(-0.0056,.0122){225}{\makebox(0,0){$\cdot$}}
\thinlines
\put(-.3,2.1){\makebox(0,0){\scriptsize $+$}}
\put(-.3,2.1){{\circle{.15}}}
\put(.3,2.1){\makebox(0,0){\scriptsize $-$}}
\put(.3,2.1){{\circle{.15}}}
\put(-.45,1.7){\makebox(0,0){\scriptsize $-$}}
\put(-.45,1.7){{\circle{.15}}}
\put(.45,1.7){\makebox(0,0){\scriptsize $+$}}
\put(.45,1.7){{\circle{.15}}}
\end{picture}
\end{center}
\caption{\label{G4}$4$th galgal $G^4$.}
\end{figure}
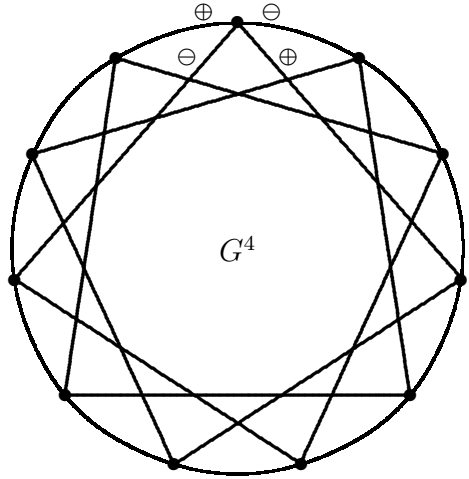

\begin{figure}
\begin{center}
\unitlength=1.5cm
\begin{picture}(6,4)(-3.00,-2)
\thicklines
\put(0,0){\hustykrouzek}
\put(0,2){\makebox(0,0){$\bullet$}} 
\multiput(0,2)(0.0097,-0.01225){200}{\makebox(0,0){$\cdot$}} 
\multiput(0,2)(0,-0.02){200}{\makebox(0,0){$\cdot$}} 
\put(0.867,1.809){\makebox(0,0){$\bullet$}} 
\multiput(0.867,1.809)(0.00348,-0.01528){200}{\makebox(0,0){$\cdot$}} 
\multiput(0.867,1.809)(-0.00867,-0.01809){200}{\makebox(0,0){$\cdot$}} 
\put(1.563,1.247){\makebox(0,0){$\bullet$}} 
\multiput(1.563,1.247)(-0.00348,-0.01528){200}{\makebox(0,0){$\cdot$}} 
\multiput(1.563,1.247)(-0.01563,-0.01247){200}{\makebox(0,0){$\cdot$}} 
\put(1.949,0.445){\makebox(0,0){$\bullet$}} 
\multiput(1.949,0.445)(-0.0097,-0.01225){200}{\makebox(0,0){$\cdot$}} 
\multiput(1.949,0.445)(-0.019492,-0.00445){200}{\makebox(0,0){$\cdot$}} 
\put(1.949,-0.445){\makebox(0,0){$\bullet$}} 
\multiput(1.949,-0.445)(- 0.01408,-0.00682){200}{\makebox(0,0){$\cdot$}} 
\multiput(1.949,-0.445)(-0.019492,0.00445){200}{\makebox(0,0){$\cdot$}} 
\put(1.563,-1.247){\makebox(0,0){$\bullet$}} 
\multiput(1.563,-1.247)(-0.0155,0){200}{\makebox(0,0){$\cdot$}} 
\multiput(1.563,-1.247)(-0.01563,0.01247){200}{\makebox(0,0){$\cdot$}} 
\put(0.867,-1.809){\makebox(0,0){$\bullet$}} 
\multiput(0.867,-1.809)( -0.01408,0.00682){200}{\makebox(0,0){$\cdot$}}
\multiput(0.867,-1.809)( -0.00867,0.01809){200}{\makebox(0,0){$\cdot$}}
\put(0,-2){\makebox(0,0){$\bullet$}} 
\multiput(0,-2)(-0.0097,0.01225){200}{\makebox(0,0){$\cdot$}} 
\put(-0.867,-1.809){\makebox(0,0){$\bullet$}} 
\multiput(-0.867,-1.809)(-0.00348,0.01528){200}{\makebox(0,0){$\cdot$}}
\put(-1.563,-1.247){\makebox(0,0){$\bullet$}} 
\multiput(-1.563,-1.247)(0.00348,0.01528){200}{\makebox(0,0){$\cdot$}} 
\put(-1.949,-0.445){\makebox(0,0){$\bullet$}} 
\multiput(-1.949,-0.445)(0.0097,0.01225){200}{\makebox(0,0){$\cdot$}} 
\put(-1.949,0.445){\makebox(0,0){$\bullet$}} 
\multiput(-1.949,0.445)(0.01408,0.00682){200}{\makebox(0,0){$\cdot$}} 
\put(-1.563,1.247){\makebox(0,0){$\bullet$}} 
\multiput(-1.563,1.247)(0.0155,0){200}{\makebox(0,0){$\cdot$}} 
\put(-0.867,1.809){\makebox(0,0){$\bullet$}} 
\multiput(-0.867,1.809)( 0.01408,-0.00682){200}{\makebox(0,0){$\cdot$}} 
\end{picture}
\end{center}
\caption{\label{G5}$5$th galgal $G^5$ (the central point is not a vertex).}
\end{figure}
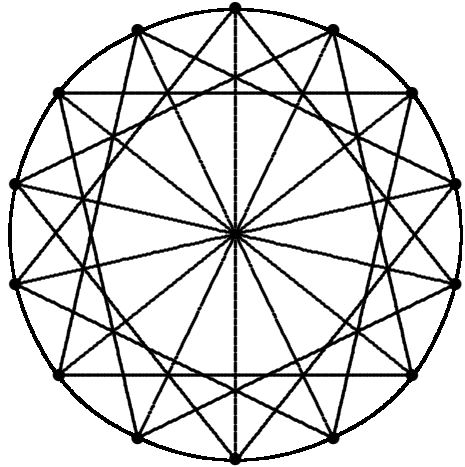

Galgalim can be used to analyze the structure of free $n$-ary
algebras. Let us, for instance, investigate possible linear dependence
of the five elementary relations~(\ref{eq:6}) among binary
bracketings~(\ref{eq:5}) of five variables. We need to solve
\begin{equation}
\label{eq:8}
 a_1 \b\!(\dve)\!\b +\ a_2 \b\!(\tri) + a_3\ \dve(\dve) +
 a_4\, (\dve)\dve + a_5\, (\tri) \b =0, 
\end{equation}
for some scalars $\Rada a15 \in \bfk$. If we view the coefficients $\Rada a15$
as decorations of the corresponding vertices of the $2$nd
galgal~(\ref{G2}), the above relation is obviously satisfied if and
only if the decorations of vertices connected by an edge agree. 
Therefore~(\ref{eq:8}) holds if and only if 
\[
a_1 = a_2 = a_3 = a_4 = a_5.
\]
The last condition is fulfilled for instance by $(\Rada a15) = (\rada
11)$, so the five elementary relations~(\ref{eq:6}) are {\em not\/}
linearly independent. This is of course elementary and well-known.

Let us proceed to the ternary case. We have eight elementary relations
which we denote, to save the space, $\rada {\R1}{\R8}$.
We consider the equation
\begin{equation}
\label{eq:9}
a_1\ \R1 + \cdots + a_8\ \R8 = 0,
\end{equation}
with some scalars $\Rada a18 \in \bfk$ which we again view as
decorations of the vertices of the $3$th galgal $G^3$. Since all the
terms in the elementary relations have the + signs,~(\ref{eq:9}) is
satisfied if and only if the decorations of two vertices connected by
an edge {\em differ\/} by the sign. The presence of closed paths of
odd lengths excludes this possibility. For instance, one has the circle
$\R1 \edge \R2\edge\R3\edge\R4\edge\R5\edge\R1$, so one requires
\[
+a_1 = -a_2 = +a_3 = -a_4 = +a_5 = -a_1
\] 
which implies $a_1 = -a_1$, therefore $a_1 = 0$ thus $a_i = 0$ for all
$1 \leq i \leq 5$. The vanishing of the remaining coefficients
in~(\ref{eq:9}) can be established in the same way.
We conclude that elementary relations for
ternary partially associative algebras are linearly independent. 
Observe that we did not need to know the labels of the
vertices and edges of the $3$th galgal explicitly, its shape was enough to
establish the linear independence of the elementary relations. We will
see in Remark~\ref{sec:free-part-assoc} 
how $G^3$ helps to understand free partially
associative $3$-algebras.

For $n=4$, axiom~(\ref{Jaruska_sibalsky_mrka.}) and thus also the
elementary relations acquire nontrivial signs. Each half-edge emerging
from a vertex of the $4$th galgal $G^4$ is therefore decorated by the
sign of the corresponding term in the relation labelling the
vertex. Explicit calculations show that this decoration obeys the rule
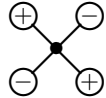
\begin{center}
\unitlength=.4cm
\begin{picture}(2,2)(0,0)
\put(1,1){\makebox(0,0){$\bullet$}}
\thicklines
\put(0.2,0.2){\line(1,1){1.6}}  
\put(1.8,.2){\line(-1,1){1.6}}  
\put(2.1,2.1){\makebox(0,0){\scriptsize $-$}}
\put(2.1,2.1){{\circle{.9}}}
\put(-.1,-.1){\makebox(0,0){\scriptsize $-$}}
\put(-.1,-.1){{\circle{.9}}}
\put(-.1,2.1){\makebox(0,0){\scriptsize $+$}}
\put(-.1,2.1){{\circle{.9}}}
\put(2.1,-.1){\makebox(0,0){\scriptsize $+$}}
\put(2.1,-.1){{\circle{.9}}}
\end{picture}
\end{center}
meaning that the antipodal half-edges acquire the same sign. It also
turns out that the decorations possesses the rotational symmetry, therefore
the decorations of all half-edges are determined by the decoration of the
half-edges adjacent to the upper vertex shown in Figure~\ref{G4}.
It is immediate to see that two half-edges of the same edge bear the
opposite signs. Therefore the elementary relations are not linearly
independent, but they, as in the binary case, sum up to zero. 

All terms of axiom~(\ref{Jaruska_sibalsky_mrka.}) and therefore
also all terms of the elementary relations for $5$-ary algebras have the
$+$ sign. As in the ternary case, their linear independence is implied by the
existence of paths of odd length in the $5$th galgal $G^5$. We leave
as an exercise to find such paths. The conclusion is that elementary
relations for $5$-ary degree $0$ partially associative algebras are
linearly independent.

\vskip .3em
\noindent 
{\bf Higher associahedra.}
Degree $0$ totally associative $n$-algebras, i.e.~algebras over the
operad $\tAss^n_0$, are, for $n \geq 1$, straightforward
generalizations of associative algebras. Observe, for instance, that
the operad $\tAss^n_0$ is, for each $n \geq 2$, the linearization of
an operad living in the monoidal category of sets and that this
property singles degree $0$ totally associative algebras out from the
four families of $n$-ary algebras introduced in
Section~\ref{Jaruska_je_moje_pusina.}. 

In~\cite{n-alg} we conjectured the existence of an analog $\KK^n =
\{\KK^n(a)\}_{a \geq 1}$ of the Stasheff associahedra for an arbitrary
$n \geq 2$. We also constructed some initial pieces of the
hypothetical $3$-associahedra $\KK^3$. It turned out that the
inductive construction contained some choices. For example, in arity
$7$ we found the following three combinatorially distinct $\KK^3(7)$'s:
\begin{center}
{
\unitlength=.3pt
\begin{picture}(200.00,220.00)(300.00,0.00)
\thicklines
\put(50.00,50.00){\makebox(0.00,0.00){$\bullet$}}
\put(0.00,100.00){\makebox(0.00,0.00){$\bullet$}}
\put(100.00,100.00){\makebox(0.00,0.00){$\bullet$}}
\put(120.00,120.00){\makebox(0.00,0.00){$\bullet$}}
\put(170.00,170.00){\makebox(0.00,0.00){$\bullet$}}
\put(150.00,150.00){\makebox(0.00,0.00){$\bullet$}}
\put(50.00,150.00){\makebox(0.00,0.00){$\bullet$}}
\put(100.00,200.00){\makebox(0.00,0.00){$\bullet$}}
\put(200.00,200.00){\makebox(0.00,0.00){$\bullet$}}
\put(200.00,0.00){\makebox(0.00,0.00){$\bullet$}}
\put(0.00,0.00){\makebox(0.00,0.00){$\bullet$}}
\put(0.00,200.00){\makebox(0.00,0.00){$\bullet$}}
\put(50.00,150.00){\line(1,0){100.00}}
\put(0.00,200.00){\line(1,-1){200.00}}
\put(0.00,0.00){\line(1,1){200.00}}
\put(0.00,200.00){\line(0,-1){200.00}}
\put(200.00,200.00){\line(-1,0){200.00}}
\put(200.00,0.00){\line(0,1){200.00}}
\put(0.00,0.00){\line(1,0){200.00}}
\put(300,0){
\put(65.00,135.00){\makebox(0.00,0.00){$\bullet$}}
\put(80.00,120.00){\makebox(0.00,0.00){$\bullet$}}
\put(40.00,60.00){\makebox(0.00,0.00){$\bullet$}}
\put(100.00,150.00){\makebox(0.00,0.00){$\bullet$}}
\put(0.00,0.00){\line(2,3){80.00}}
\put(100.00,100.00){\makebox(0.00,0.00){$\bullet$}}
\put(170.00,170.00){\makebox(0.00,0.00){$\bullet$}}
\put(150.00,150.00){\makebox(0.00,0.00){$\bullet$}}
\put(50.00,150.00){\makebox(0.00,0.00){$\bullet$}}
\put(200.00,200.00){\makebox(0.00,0.00){$\bullet$}}
\put(200.00,0.00){\makebox(0.00,0.00){$\bullet$}}
\put(0.00,0.00){\makebox(0.00,0.00){$\bullet$}}
\put(0.00,200.00){\makebox(0.00,0.00){$\bullet$}}
\put(50.00,150.00){\line(1,0){100.00}}
\put(0.00,200.00){\line(1,-1){100.00}}
\put(0.00,0.00){\line(1,1){200.00}}
\put(0.00,200.00){\line(0,-1){200.00}}
\put(200.00,200.00){\line(-1,0){200.00}}
\put(200.00,0.00){\line(0,1){200.00}}
\put(0.00,0.00){\line(1,0){200.00}}
}
\put(600,0){
\put(140.00,60.00){\makebox(0.00,0.00){$\bullet$}}
\put(20.00,120.00){\makebox(0.00,0.00){$\bullet$}}
\put(70.00,70.00){\makebox(0.00,0.00){$\bullet$}}
\put(200.00,100.00){\makebox(0.00,0.00){$\bullet$}}
\put(100.00,200.00){\makebox(0.00,0.00){$\bullet$}}
\put(160.00,160.00){\makebox(0.00,0.00){$\bullet$}}
\put(100.00,100.00){\makebox(0.00,0.00){$\bullet$}}
\put(40.00,40.00){\makebox(0.00,0.00){$\bullet$}}
\put(0.00,0.00){\makebox(0.00,0.00){$\bullet$}}
\put(200.00,0.00){\makebox(0.00,0.00){$\bullet$}}
\put(200.00,200.00){\makebox(0.00,0.00){$\bullet$}}
\put(0.00,200.00){\makebox(0.00,0.00){$\bullet$}}
\put(0.00,200.00){\line(4,-1){160.00}}
\put(0.00,200.00){\line(1,-4){40.00}}
\put(200.00,0.00){\line(-1,1){100.00}}
\put(0.00,0.00){\line(1,1){200.00}}
\put(200.00,0.00){\line(-1,0){200.00}}
\put(200.00,200.00){\line(0,-1){200.00}}
\put(0.00,200.00){\line(1,0){200.00}}
\put(0.00,0.00){\line(0,1){200.00}}
}
\end{picture}}
\end{center}
They are convex $2$-dimensional polyhedra with twelve vertices,
sixteen edges and five $2$-dimen\-sional faces.  We refer to
\cite{n-alg} for more details.

\section{Koszulness - the case study}
\label{JarunKa}

This section is devoted to the following statement organized in
the table of Figure~\ref{0}.

\begin{figure}
\begin{center}
{
\unitlength=0.085em
\begin{picture}(250.00,200.00)(0.00,0.00)
\put(25.00,55.00){\makebox(0.00,0.00)[b]{$\pAss^n_d$}}
\put(25.00,135.00){\makebox(0.00,0.00)[b]{$\tAss^n_d$}}
\put(25.00,95.00){\makebox(0.00,0.00)[b]{$\ttildeAss^n_d$}}
\put(25.00,15.00){\makebox(0.00,0.00)[b]{$\ptildeAss^n_d$}}
\put(70.00,06.60){\makebox(0.00,0.00)[b]{$d$ odd}}
\put(70.00,46.60){\makebox(0.00,0.00)[b]{$d$ odd}}
\put(70.00,86.60){\makebox(0.00,0.00)[b]{$d$ odd}}
\put(70.00,126.60){\makebox(0.00,0.00)[b]{$d$ odd}}
\put(70.00,26.60){\makebox(0.00,0.00)[b]{$d$ even}}
\put(70.00,66.60){\makebox(0.00,0.00)[b]{$d$ even}}
\put(70.00,106.60){\makebox(0.00,0.00)[b]{$d$ even}}
\put(70.00,146.60){\makebox(0.00,0.00)[b]{$d$ even}}
\put(210.00,186.60){\makebox(0.00,0.00)[b]{$n > 7$}}
\put(230.00,26.60){\makebox(0.00,0.00)[b]{?}}
\put(190.00,26.60){\makebox(0.00,0.00)[b]{?}}
\put(190.00,46.60){\makebox(0.00,0.00)[b]{?}}
\put(230.00,66.60){\makebox(0.00,0.00)[b]{?}}
\put(230.00,86.60){\makebox(0.00,0.00)[b]{?}}
\put(190.00,106.60){\makebox(0.00,0.00)[b]{?}}
\put(230.00,126.60){\makebox(0.00,0.00)[b]{?}}
\put(190.00,126.60){\makebox(0.00,0.00)[b]{?}}
\put(150.00,26.60){\makebox(0.00,0.00)[b]{no}}
\put(110.00,26.60){\makebox(0.00,0.00)[b]{no}}
\put(110.00,46.60){\makebox(0.00,0.00)[b]{no}}
\put(150.00,66.60){\makebox(0.00,0.00)[b]{no}}
\put(150.00,86.60){\makebox(0.00,0.00)[b]{no}}
\put(110.00,106.60){\makebox(0.00,0.00)[b]{no}}
\put(150.00,126.60){\makebox(0.00,0.00)[b]{no}}
\put(110.00,126.60){\makebox(0.00,0.00)[b]{no}}
\put(230.00,05.00){\makebox(0.00,0.00)[b]{yes}}
\put(190.00,05.00){\makebox(0.00,0.00)[b]{yes}}
\put(150.00,05.00){\makebox(0.00,0.00)[b]{yes}}
\put(110.00,05.00){\makebox(0.00,0.00)[b]{yes}}
\put(230.00,45.00){\makebox(0.00,0.00)[b]{yes}}
\put(190.00,65.00){\makebox(0.00,0.00)[b]{yes}}
\put(150.00,45.00){\makebox(0.00,0.00)[b]{yes}}
\put(110.00,65.00){\makebox(0.00,0.00)[b]{yes}}
\put(230.00,105.00){\makebox(0.00,0.00)[b]{yes}}
\put(190.00,85.00){\makebox(0.00,0.00)[b]{yes}}
\put(150.00,105.00){\makebox(0.00,0.00)[b]{yes}}
\put(230.00,145.00){\makebox(0.00,0.00)[b]{yes}}
\put(190.00,145.00){\makebox(0.00,0.00)[b]{yes}}
\put(150.00,145.00){\makebox(0.00,0.00)[b]{yes}}
\put(110.00,85.00){\makebox(0.00,0.00)[b]{yes}}
\put(110.00,145.00){\makebox(0.00,0.00)[b]{yes}}
\put(230.00,166.60){\makebox(0.00,0.00)[b]{$n$ odd}}
\put(150.00,166.60){\makebox(0.00,0.00)[b]{$n$ odd}}
\put(190.00,166.60){\makebox(0.00,0.00)[b]{$n$ even}}
\put(110.00,166.60){\makebox(0.00,0.00)[b]{$n$ even}}
\put(130.00,190.00){\makebox(0.00,0.00){$n \leq 7$}}
\put(50.00,20.00){\line(1,0){200.00}}
\put(50.00,60.00){\line(1,0){200.00}}
\put(50.00,100.00){\line(1,0){200.00}}
\put(50.00,140.00){\line(1,0){200.00}}
\put(210.00,180.00){\line(0,-1){180.00}}
\put(130.00,180.00){\line(0,-1){180.00}}
\put(90.00,180.00){\line(1,0){160.00}}
\put(170.00,200.00){\line(0,-1){200.00}}
\put(90.00,160.00){\line(0,-1){160.00}}
\thicklines
\put(20.00,120.00){\line(1,0){30.00}}
\put(20.00,80.00){\line(1,0){30.00}}
\put(20.00,40.00){\line(1,0){30.00}}
\put(20.00,0.00){\line(1,0){30.00}}
\put(50.00,160.00){\line(-1,0){30.00}}
\put(250.00,170.00){\line(0,-1){10.00}}
\put(250.00,200.00){\line(0,-1){30.00}}
\put(90.00,200.00){\line(1,0){160.00}}
\put(90.00,160.00){\line(0,1){40.00}}
\put(50.00,40.00){\line(1,0){200.00}}
\put(50.00,80.00){\line(1,0){200.00}}
\put(50.00,120.00){\line(1,0){200.00}}
\put(50.00,0.00){\line(0,1){160.00}}
\put(240.00,0.00){\line(-1,0){190.00}}
\put(250.00,0.00){\line(-1,0){10.00}}
\put(250.00,160.00){\line(0,-1){160.00}}
\put(50.00,160.00){\line(1,0){200.00}}
\put(0.00,120.00){\line(1,0){20.00}}
\put(0.00,80.00){\line(1,0){20.00}}
\put(0.00,40.00){\line(1,0){20.00}}
\put(0.00,0.00){\line(1,0){20.00}}
\put(0.00,160.00){\line(0,-1){160.00}}
\put(20.00,160.00){\line(-1,0){20.00}}
\end{picture}}

\end{center}
\caption{\label{0}
 Koszulness of the operads $\tAss^n_d$,
$\pAss^n_d$, $\ttildeAss^n_d$ and $\ptildeAss^n_d$.
``Yes'' means that the corresponding operad is Koszul,
``no'' that it is not Koszul.}
\end{figure}
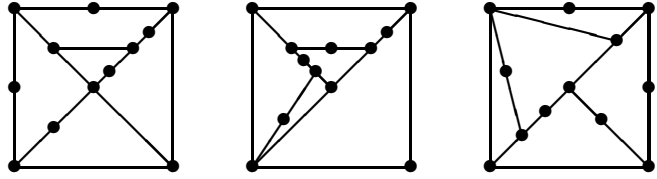

\begin{theorem}
\label{2}
Let $n \leq 7$. Then 
the operad $\tAss^n_d$ is Koszul if and only if $d$ is even.  The
operad $\pAss^n_d$ is Koszul if and only if $n$ and $d$ have the same
parity. The operad $\ttildeAss^n_d$ is Koszul if and only if $n$ and $d$ have
different parities.
The operad $\ptildeAss^n_d$ is Koszul if and only if $d$ is
odd. 

The operads  $\tAss^n_d$ with $d$ even, $\pAss^n_d$ with $n$ and $d$
of the same parity, $\ttildeAss^n_d$ with  $n$ and $d$ of different
parities, and  $\ptildeAss^n_d$ with $d$ odd, are Koszul for all $n \geq 2$.
\end{theorem}
  
The Koszulness part (``yes'' in the table of Figure~\ref{0}) will
follow from~\cite{hoffbeck} and relations between the operads
$\tAss^n_d$, $\pAss^n_d$, $\ttildeAss^n_d$ and $\ptildeAss^n_d$, see
Proposition~\ref{a}. The non-Koszulness part (``no'' in
Figure~\ref{0}) will, for $n \leq 7$, follow in a similar fashion from
Proposition~\ref{5}. We do not know how to extend our proof of
Proposition~\ref{5} for $n \geq 8$, we therefore put question marks
to the corresponding places in Figure~\ref{0}. See also Remark~\ref{JARuska}
and the first problem of Section~\ref{open}.  

In particular, the operads $\tildeAss$ and $\tAss^2_1 = \pAss^2_1$
are not Koszul. Let us formulate useful

\begin{lemma}
\label{o}
Let $\calP^n_d$ be one of the operads above. Then $\calP^n_d$ is
Koszul if and only if $\calP^n_{d+2}$ is Koszul, that is, only the
parity of $d$ matters.
\end{lemma}

\begin{proof}
There is a `twisted'
isomorphism
\begin{equation}
\label{9}
\varphi : \calP^n_d \stackrel\cong\longrightarrow  \calP^n_{d+2},
\end{equation}
i.e.~a sequence of equivariant isomorphisms $\varphi(a): 
\calP^n_d(a)\to \calP^n_{d+2}(a)$, $a \geq 1$, that commute with the
$\circ_i$-operations such that the component 
$\varphi(k(n-1)+1)$ is of degree
$2k$, for $k \geq 0$. 

To construct such an isomorphism, consider an operation $\mu'$ of
arity $n$ and degree $d$, and another operation $\mu''$ of the same arity but
of degree $d + i$. We leave as an exercise to verify that
the assignment $\mu' \mapsto \mu''$ extends to a twisted isomorphism
$\omega: \Gamma(\mu') \to \Gamma(\mu'')$ if and only if $i$ is even. 

Let $\calP^n_d = \Gamma(\mu')/(R')$ and  $\calP^n_{d+2} =
\Gamma(\mu'')/(R'')$. It is clear that the twisted isomorphism $\omega :
\Gamma(\mu') \to \Gamma(\mu'')$ preserves the ideals of relations, so
it induces a twisted isomorphism~(\ref{9}). A moment's reflection
convinces one that $\varphi$ induces similar twisted
isomorphisms of the Koszul duals and the bar constructions. 
This, by Definition~\ref{Zitra_budu_s_Jaruskou!}, gives the lemma.  
\end{proof}

\begin{proposition}
\label{a}
The operads marked ``yes'' in the tables of Figure~\ref{0} are Koszul.
\end{proposition}

\begin{proof}
The operads $ \tAss^n_0$ are Koszul for all $n \geq 2$
by~\cite[\S~7.2]{hoffbeck} 
(see also~\cite{gnedbaye:CM96} for the case $n$ even and $d=0$). 
So, by Lemma~\ref{o}, the operads 
$ \tAss^n_d$ are Koszul for all even $d$ and $n \geq 2$, which gives the
four ``yes'' in the first row of the table in Figure~\ref{0}.

The ``yes'' in the 3rd row follow from the ``yes'' in the
1st row, the fact that an operad is Koszul if and only if its
dual operad is Koszul proved in Proposition~\ref{sec:dual-quadr-oper},
and the isomorphism $(\pAss^n_d)^!= \tAss^n_{-d+n-2}$ established in
Proposition~\ref{JarKA}.  The ``yes'' in the remaining rows in
Figure~\ref{0} follow from the ``yes'' in the 1st and the 3rd rows, and 
Proposition~\ref{Jaruska_sbira_houbicky} by which
the suspension preserves Koszulness.
\end{proof}

The ``no'' entries in Figure~\ref{0} will be established using the
Ginzburg-Kapranov criterion~(\ref{zitra_Jarka}).
Our first task will therefore be to describe the Poincar\'e series of
the family $\tAss^n_d$ which generates, via the duality and suspension, all
the remaining operads.

\begin{lemma}
\label{l}
The generating function for the operad $\tAss_d^n$ is 
\begin{equation}
\label{Zitra_ma_priletet_Jarka_do_Mulhouse}
g_{\sstAss_d^n}(t) := 
\cases{\displaystyle\frac t{1-t^{n-1}}}{if $d$  is even, and}
{t-t^{n}+t^{2n-1}}{if $d$ is  odd.}
\end{equation}
\end{lemma}

\begin{proof}
The components of the operad $\tAss_d^n$ are trivial in arities
different from $k(n-1) +1$, $k \geq 0$. The piece $\tAss_d^n(k(n-1)
+1)$ is generated by all possible $\circ_i$-compositions involving $k$
instances of the generating operation $\mu$, modulo the relations
\begin{equation}
\label{,}
\mu \circ_i\mu- \mu \circ_j\mu, 
\ {\rm for} \ i,j=1,\ldots ,n
\end{equation}
which enable one to replace each $\mu \circ_i \mu$, $2 \leq i \leq n$,
by $\mu \circ_1 \circ \mu$.

If the degree $d$ is {\em even\/}, the operad $\tAss_d^n$ is evenly graded,
so the associativity~\cite[p.~1473, Eqn.~(1)]{markl:zebrulka} of the
$\circ_i$-operations does not involve signs. Therefore an arbitrary
$\circ_i$-composition of $k$ instances of $\mu$ can be brought to the
form
\[
\eta_k :=(\cdots((\mu \circ_1 \mu) \circ_1 \mu) \circ_1  \cdots ) \circ_1 \mu.
\]
We see that  $\tAss_d^n(k(n-1) +1)$ is spanned by the set 
$\{\eta_k \circ \sigma;\ \sigma \in \Sigma_{k(n-1) +1}\}$, so 
\[
\dim(\tAss_d^n(k(n-1) +1)) =
(k(n-1) +1)!
\] 
and, by definition,
\[
g_{\sstAss_d^n}(t) = \sum_{k \geq 0}  t^{k(n-1)+1}= 
\displaystyle\frac t{1-t^{n-1}},
\] 
which verifies the even case of~(\ref{Zitra_ma_priletet_Jarka_do_Mulhouse}).

The {\em odd\/} case is subtler since the
associativity~\cite[p.~1473, Eqn.~(1)]{markl:zebrulka} may involve nontrivial
signs. As in the even case we calculate that 
\begin{equation}
\label{C}
\dim(\tAss_d^n(k(n-1) +1)) =
(k(n-1) +1)!\ \mbox { for } k = 0,1,2,
\end{equation} 
because these small arities do not require the associativity.

If $k \geq 3$, we can still to bring 
each $\circ_i$-composition of $k$
instances of $\mu$ to the form of the `canonical' generator 
$\eta_k$, but we may get a
nontrivial sign which may moreover depend on the way we applied the
associativity. Relation~(\ref{,}) implies that
\begin{equation}
\label{A}
(\mu \circ_1 \mu) \circ_1 \mu = (\mu \circ_n \mu)\circ_1 \mu
\end{equation}
in  $\tAss_d^n(3n-2)$.
Applying~(\ref{,}) and the
associativity~\cite[p.~1473, Eqn.~(1)]{markl:zebrulka} several times,
we get that
\begin{align}\nonumber 
(\mu \circ_1 \mu) \circ_1 \mu & = \mu \circ_1( \mu  \circ_1 \mu) =
\mu \circ_1 (\mu  \circ_n \mu) = (\mu \circ_1 \mu) \circ_n \mu
=  (\mu \circ_n \mu) \circ_n \mu
\\
\label{B}
& = \mu \circ_n( \mu  \circ_1 \mu) =
 \mu \circ_n( \mu  \circ_n \mu) = (\mu \circ_n \mu) \circ_{2n-1} \mu
\\ \nonumber 
& = (\mu \circ_1 \mu) \circ_{2n-1} \mu.
\end{align}
Since the degree of $\mu$ is odd, the first line of the
associativity~\cite[p.~1473, Eqn.~(1)]{markl:zebrulka} implies
\[
(\mu \circ_1 \mu) \circ_{2n-1} \mu = -  (\mu \circ_n \mu) \circ_1 \mu
\]
therefore~(\ref{A}) and~(\ref{B}) combine into
\[
(\mu \circ_1 \mu) \circ_1 \mu  =- (\mu \circ_1 \mu) \circ_1 \mu.
\]
This means that $(\mu \circ_1 \mu) \circ_1 \mu = 0$ so
$\tAss_d^n(3n-2)= 0$. Since $\tAss_d^n(k(n-1) +1)$ is, for $k \geq 3$,
generated by $\tAss_d^n(3n-2)$, we conclude that $\tAss_d^n(k(n-1)
+1) = 0$ for $k \geq 3$ which, along with~(\ref{C}), verifies the odd 
case of~(\ref{Zitra_ma_priletet_Jarka_do_Mulhouse}).
\end{proof}

\begin{remark}
{\rm
The Poincar\'e series of an operad $\calP$ and its suspension
$\osusp \calP$ are related by
$g_{{\ssosusp}\calP}(t)=-g_\calP(-t)$. Lemma~\ref{l} thus implies that
the generating series of the operad $\ttildeAss^n_d =
\osusp\tAss^n_{d-n+1}$ 
equals 
\[
g_{\ssttildeAss^n_d}(t) := 
\cases
{t + (-1)^d t^n + t^{2n-1}}{if $n$ and $d$ have the same parity, and}
{\rule{0pt}{1.8em}\displaystyle\frac t{1-(-1)^d t^{n-1}}}
{if $n$ and $d$ have different parities.}
\]
We do not know explicit formulas for the Poincar\'e series of
$\pAss^n_d$ and $\ptildeAss^n_d$ 
except in the case $n =2$ when these operads coincide with
the corresponding (anti)-associative operads.
}\end{remark}

\begin{example}
\label{Pozitri_uvidim_Jarusku.}
{\rm
It easily follows from the above calculations that, for the
anti-associative operad $\tildeAss$, one has
\[
\tildeAss(1) \cong \bfk,\
\tildeAss(2) \cong \bfk[\Sigma_2] \mbox { and }
\tildeAss(3) \cong \bfk[\Sigma_3],
\]
while $\tildeAss(a) = 0$ for $a \geq 4$.
}
\end{example}

Let us return to our task of proving the non-Koszulness of the ``no''
cases in the tables of Figure~\ref{0}.
Our strategy  will be to interpret~(\ref{zitra_Jarka}) as saying
that $-g_{\calP^!}(-t)$ is a formal inverse of $g_{\calP}(t)$ at $0$.
Since $g'_{\calP}(0) = 1$, this unique formal inverse
exists.  In the particular case of $\calP = \tAss^n_d$, with $d$ odd, 
this means that
$-g_{\sspAss^n_{-d+n-2}}(-t)$ should be compared to a formal inverse of
$g_{\sstAss^n_d}(t) = t - t^n + t^{2n-1}$.  A simple degree count shows
that $g_{\sspAss^n_{-d+n-2}}(t)$ is of the form
\[
\cases{t - A_1t^n + A_2 t^{2n-1} - A_3 t^{3n-2} + \cdots}{for $n$ even
  and}
  {\rule{0pt}{1em}t + A_1t^n + A_2 t^{2n-1} + A_3 t^{3n-2} + \cdots}
{for $n$ odd,}
\]
for some {\em non-negative\/} integers $A_1,A_2,A_3,\ldots$, therefore 
$-g_{\sspAss^n_{-d+n-2}}(-t)$ is in both cases the formal power series
\begin{equation}
\label{v}
t + A_1t^n + A_2 t^{2n-1} + A_3 t^{3n-2} + \cdots
\end{equation}
with non-negative coefficients. If we show that the formal inverse of
$t - t^n + t^{2n-1}$ is not of this form, by
Theorem~\ref{gk} the corresponding operad $\tAss_d^n$ is not Koszul.

\begin{example}{\rm
The Poincar\'e series of the operad
$\tAss^2_1$ is, by Lemma~\ref{l},
\[
g_{\sstAss_1^2}(t) =  t-t^2+t^3.
\]
One can compute the formal inverse of this function as 
\[
t+t^2+t^3-4t^5-14t^6-30t^7-33t^8+55t^9+ \cdots.
\]
The presence of negative coefficients implies that the operad
$\tAss^2_1$ is not Koszul, neither is the anti-associative operad 
$\tildeAss = \ttildeAss^2_0 = \osusp^{-1} \tAss^2_1$.

Likewise, the Poincar\'e series of the operad
$\tAss^3_1$ equals
\[
g_{\sstAss_1^3}(t) =  t - t^3+t^5
\]
and we computed, using {\tt Matematica}, the initial part of the
formal inverse as
\[
t+t^3+2t^5+4t^7+ 5t^9-13t^{11}-147t^{13}+\cdots
\]
The existence of negative coefficients again implies that the operad
$\tAss^3_1$ is not Koszul.  The formal inverse of
\[
g_{\sstAss_1^4}(t) =  t - t^4+t^7
\]
up to the first negative term is 
\[
t+t^4+3t^7+11t^{10}+ 42t^{13}+153t^{16}+469t^{19}+690t^{22}-5967t^{25}+\cdots
\]
so $\tAss^4_1$ is not Koszul.
}\end{example}

The complexity of the calculation of the relevant initial part of the
inverse of $g_{\sstAss^n_1}(t) = t-t^{n}+t^{2n-1}$ 
grows rapidly with $n$. We have, however, the following:

\begin{proposition}
\label{5}
For $n \leq 7$, the formal inverse of $t-t^{n}+t^{2n-1}$ has at least
one negative coefficient. Therefore the operads $\tAss^n_d$ for $d$
odd and $n\leq 7$ are not Koszul. 
\end{proposition}

\begin{proof}
The function $g(z) := z-z^{n}+z^{2n-1}$ is analytic in the
complex plane $\bbbC$. Its analytic inverse $g^{-1}(z)$ is a
not-necessarily single-valued analytic function defined outside the
points in which the derivative $g'(z)$ vanishes. Let us denote by
$\frakZ$ the set of these points, i.e.
\[
\frakZ := \{z \in \bbbC;\ g'(z) = 0\}.
\]
The key observation is that, for $n \leq 7$, the equation $g'(z) = 0$
has no real solutions, $\frakZ \cap \bbbR = \emptyset$. Indeed, one
has to solve the equation
\begin{equation}
\label{Dnes_prileti_Jarka_do_Mulhouse!}
g'(z) = 1 - n z^{n-1} + (2n-1)z^{2n-2} = 0
\end{equation}
which, after the substitution $w := z^{n-1}$ leads to the quadratic
equation
\[
1 - nw + (2n-1)w^2 = 0
\]
whose discriminant $n^2 - 8n + 4$ is, for $n \leq 7$, negative.

Let $f(z)$ be the power series representing the 
branch at $0$ of $g^{-1}(z)$ such that $f(0) = 0$. It is clear that
$f(t)$ is precisely the formal inverse of $g(t)$ at $0$.
Suppose that 
\[
f(z) = z + a_2 z^2 + a_3 z^3 + a_4z^4 +\cdots,
\]
with all coefficients $a_2,a_3,a_4,\ldots$ non-negative real numbers.
Since $\frakZ \not= \emptyset$ and obviously $0 \not \in \frakZ$, 
the radius of convergence of $f(z)$ at $0$, which equals the radius of
the maximal circle centered at $0$ whose interior does not contain
points in $\frakZ$, is
some number $r$ with $0 < r < \infty$. Let ${\mathfrak z} \in \frakZ$
be such that $|{\mathfrak z}| = r$. Since all coefficients
of the power series $f$ are positive, we have
\[
|f({\mathfrak z})| \leq f(|{\mathfrak z}|) = f(r),
\] 
so the function $f(r)$ must have singularity at the {\em real\/} point
$r \in \bbbR$, i.e.~$g'(z)$ must vanish at~$r$. This contradicts the
fact that $g'(z) = 0$ has no real solutions.
\end{proof}

\begin{remark}
\label{JARuska}
{\rm
Equation~(\ref{Dnes_prileti_Jarka_do_Mulhouse!}) has, for $n=8$, two
real solutions, ${\mathfrak z}_1 = \sqrt[7]{1/3}$ and ${\mathfrak z}_2
= \sqrt[7]{1/5}$. This means that the inverse function of
$z-z^{n}+z^{2n-1}$ has two positive real poles and the arguments used
in our proof of Proposition~\ref{5} do not apply.

We verified Proposition~\ref{5} using {\tt Matematica}. The first
negative coefficient in the inverse of $t-t^{n}+t^{2n-1}$
was at the power $t^{57}$ for $n=5$, at $t^{161}$ for $n=6$, and
at $t^{1171}$ for $n=7$. 
For $n=8$ we did not  find any negative term
of degree less than $10~000$. It is indeed possible that all coefficients
of the inverse of $t-t^8+t^{15}$ are positive.
}
\end{remark}

Proposition~\ref{5} together with the fact that the suspension and the
!-dual preserves Koszulness (Propositions~\ref{Jaruska_sbira_houbicky}
and \ref{sec:dual-quadr-oper}) imply the ``no'' entries of the tables
in Figure~\ref{0} for $n \leq 7$.

\section{Cohomology of algebras over non-Koszul operads -- an example}
\label{sec:cohom-algebr-over}

In this section we study anti-associative algebras introduced in
Definition~\ref{Jar}, i.e.  structures $A = (V,\mu)$ with a degree-$0$
bilinear anti-associative multiplication \hbox{$\mu : \otexp V2 \to
V$}.  We describe the `standard' cohomology $\Hst*$ of an
anti-associative algebra $A$ with coefficients in itself and compare
it to the relevant part of the deformation cohomology $\H*$ based on
the minimal model of the anti-associative operad $\tildeAss$. Since
$\tildeAss$ is, by Theorem~\ref{2}, not Koszul, these two cohomologies
differ. While the standard cohomology has no sensible meaning, the
deformation cohomology coincides with the triple
cohomology~\cite{fox:JPAA93,fox-markl:ContM97} and governs
deformations of anti-associative algebras.

\begin{examples}
{\rm
Anti-associative algebras, as algebras over a non-Koszul operad, should
possess a lot of peculiar properties. Therefore, due to the `anthropic
principle,' one can hardly expect to find examples of these structures
in Nature. Observe, however, that there still are `natural' examples of the
anti-associativity. For instance, the standard basis elements
$\{\Rada e18\}$ of the octonions (also called the  Cayley algebra)
satisfy
\[
(e_ie_j)e_k=-e_i(e_je_k),
\]
whenever $e_i e_j \neq e_k$ and $1 \leq i,j,k\leq 8$ are  distinct.

Since $\tildeAss(a)=0$ for $a\geq 4$, the product of four elements in
an arbitrary anti-associative algebra is trivial. Anti-associative
algebras are therefore always 3-step nilpotent. Below we classify,  for $k \leq 3$,
isomorphism classes of anti-associative structures on the
$k$-dimensional vector space $V: = \Span(\Rada e1k)$.

\noindent 
{\em Case $k=1$.} The only $1$-dimensional anti-associative algebra is the
  trivial one, with $e_1 \cdot e_1=0$.

\noindent 
{\em Case $k=2$.} In dimension $2$, there are two non-isomorphic
anti-associative algebras: the trivial one, and the one defined by $e_1
\cdot e_1=e_2$ and the remaining products of the basis elements trivial.

\noindent 
{\em Case $k=3$.}
In dimension $3$, we distinguish two subclasses of anti-associative
algebras. Algebras in the first subclass satisfy $v \cdot v=0$ for all
$v \in V$. There are two non-isomorphic algebras in this subclass,
the trivial one, and the one with $e_1 \cdot e_2=-e_2 \cdot e_1=e_3$
and the remaining products of the basic elements trivial.

Algebras in the second subclass contain some $v$ with
$v \cdot v\neq 0$.  Algebras with this property are either 
isomorphic to the one given by:
$$\left\{
\begin{array}{l}
e_1 \cdot e_1=e_2,\\
e_1 \cdot e_2=-e_2 \cdot e_1=e_3 ,
\end{array}
\right.$$
which happens to be the free anti-associative algebra on one
generator,
or to an algebra belonging to one of the following  
two $2$-dimensional families:
\[
\left\{
\begin{array}{l}
e_1 \cdot e_1=e_2,\\
e_1 \cdot e_3=ae_2, \\
e_3 \cdot e_1=be_2 ,\\
e_3 \cdot e_3=e_2,\
\end{array}
\right.
\quad
\left\{
\begin{array}{l}
e_1 \cdot e_1=e_2,\\
e_1 \cdot e_3=ae_2, \\
e_3 \cdot e_1=be_2 ,\\
\end{array}
\right.
\]
where $a,b \in \bfk$.
}\end{examples}

Let us return to the main construction of this section. 
It was explained at several
places~\cite{markl:JPAA96,markl:zebrulka,markl:ib,markl:ba} how a,
not-necessarily acyclic, quasi-free resolution $(\calP,\pa=0)
\stackrel{\rho}{\longleftarrow} (\scrR,\pa)$ of an operad $\calP$, which
we assume for simplicity non-dg and concentrated in degree $0$,
determines a cohomology theory for $\calP$-algebras with coefficients
in itself. If $\calP$ is quadratic and if we take as $(\scrR,\pa)$ the
dual bar construction (recalled in
Section~\ref{Za_chvili_jdu_za_Jaruskou.}) 
of the quadratic dual $\calP^!$, we get the
`standard' cohomology $\PHst*$ as the cohomology of the
`standard' cochain complex
\[
\PCst1 \stackrel{\dst^1}{\longrightarrow} \PCst2
\stackrel{\dst^2}{\longrightarrow}
\PCst3 \stackrel{\dst^3}{\longrightarrow}
\PCst4 \stackrel{\dst^4}{\longrightarrow} \cdots
\]
in which $\PCst p := \Hom(\calP^!(p) \ot_{\Sigma_p} \otexp Vp,V)$, $p
\geq 1$, and the differential $\dst^*$ is induced from the structure
of $\calP^!$ and $A$, see~\cite[Section~8]{fox-markl:ContM97} 
or~\cite[Definition~II.3.99]{markl-shnider-stasheff:book}.  
This type of (co)homology was
considered in the seminal paper~\cite{ginzburg-kapranov:DMJ94}.

The deformation (also called, in~\cite{markl:JPAA96}, the {\em cotangent\/})
cohomology uses the minimal model of $\calP$ in place of
$(\scrR,\pa)$. Recall~\cite[p.~1479]{markl:zebrulka} that the {\em minimal
model\/} of an operad $\calP$ is a homology isomorphism
\[
(\calP,0) \stackrel{\rho}{\longleftarrow} (\Gamma (M),\partial)
\] 
of dg-operads such that the image of $\partial$ consists of
decomposable elements of the free operad $\Gamma (M)$ (the
minimality). It is
known~\cite[Section~II.3.10]{markl-shnider-stasheff:book} that each
operad with $\calP(1) \cong \bfk$ admits a minimal model unique up to
isomorphism. The {\em deformation cohomology\/} $\PH*$ is the
cohomology of the complex
\[
\PC1 \stackrel{\delta^1}{\longrightarrow} \PC2
\stackrel{\delta^2}{\longrightarrow}
\PC3 \stackrel{\delta^3}{\longrightarrow}
\PC4 \stackrel{\delta^4}{\longrightarrow} \cdots
\]
in which $\PC1 := \Hom(V,V)$ and 
\[
\PC p := \Hom(\textstyle\bigoplus_{q\geq 2} E_{p-2}(q) \ot_{\Sigma_q}
\otexp Vq,V), \ \mbox { for } p \geq 2.
\] 
The differential $\delta^*$ is defined by the formula which can be
found in~\cite[Section~2]{markl:ib} or in the introduction
to~\cite{markl:ba}.  If $\calP$ is quadratic Koszul, the dual bar
construction of $\calP^!$ is,
by~\cite[Proposition~2.6]{markl:zebrulka}, isomorphic to the minimal
model of $\calP$, thus the standard and deformation cohomologies
coincide, giving rise to the `standard' constructions such as the
Hochschild, Harrison or Chevalley-Eilenberg cohomology.

Neither $\PHst*$ nor $\PH*$ have the $0$th term. A natural $H^0$
exists only for algebras for which the concept of unitality makes
sense. This is not always the case. Assume, for example, that an
anti-associative algebra $A = (V,\mu)$ has a unit, i.e.~and element $1
\in V$ such that $1a = a1 = a$, for all $a \in V$. Then the
anti-associativity~(\ref{Jaruska_je_pusinka}) with $c = 1$ gives $ab +
ab = 0$, so $ab = 0$ for each $a,b \in V$.

Let us describe the standard cohomology $\Hst*$ of an
anti-associative algebra $A = (V,\mu)$. The operad $\tildeAss$ is,
by Proposition~\ref{JarKA}, self-dual and it follows from the 
description of $\tildeAss = \tildeAss{}^!$ given in
Example~\ref{Pozitri_uvidim_Jarusku.} that
$\Hst*$ is the cohomology of
\[
\Cst1 \stackrel{\dst^1}{\longrightarrow} \Cst2
\stackrel{\dst^2}{\longrightarrow}
\Cst3 \stackrel{\dst^3}{\longrightarrow}
0 \stackrel{0}{\longrightarrow} 0 \stackrel{0}{\longrightarrow} \cdots
\]
in which $\C p := \Hom(\otexp Vp,V)$ for $p = 1,2,3$, and all higher
$\C p$'s are trivial. The two nontrivial pieces of the differential are
basically the Hochschild differentials with ``wrong'' signs of some
terms:
\begin{align*}
\delta^1(\varphi)(a,b)& := a \varphi(b) - \varphi(ab) + \varphi(a)b,
\mbox { and}
\\
\delta^2(f)(a,b,c) &:= a f(b,c)+ f(ab,c)+ f(a,bc) +f(a,b)c,
\end{align*}
for $\varphi \in \Hom(V,V)$, $f \in \Hom(\otexp V2,V)$ and $a,b,c \in
V$. We abbreviated $\mu(a,b) = ab$, $\mu(a,\varphi(b)) = a\varphi(b)$,
\&c. One sees, in particular, that $\Hst p = 0$ for $p\geq 4$.

Let us describe the relevant part of the deformation
cohomology of $A$. It can be shown that  
$\tildeAss$ has the minimal model
\[
(\tildeAss,0) \stackrel{\rho}{\longleftarrow} (\Gamma (E),\partial)
\] 
with the generating $\Sigma$-module $E = \{E(a)\}_{a \geq 2}$ such that 
\begin{itemize}
\item[--]
$E(2)$ is generated by a degree $0$ bilinear operation $\mu_2 : V
   \ot V \to V$,
\item[--]
$E(3)$ is generated by a degree $1$ trilinear operation $\mu_3 :
   \otexp V3 \to V$, 
\item[--]
$E(4)=0$, and
\item[--]
$E(5)$ is generated by four $5$-linear degree $2$ operations
   $\mu_5^1,\mu_5^2,\mu_5^3,\mu_5^4 : \otexp V5 \to V$, 
\end{itemize}
so the minimal model of $\tildeAss$ is of the form
\[
(\tildeAss,0) \stackrel{\alpha}{\longleftarrow} (\Gamma(\mu_2,\mu_3,
\mu_5^1,\mu_5^2,\mu_5^3,\mu_5^4,\ldots),\pa).
\]
Notice the gap in the arity $4$ generators! We do not know the exact
form of the pieces $E(a)$, $a \geq 6$, of the generating
$\Sigma$-module $E$, but we know that they do not contain elements of
degrees $\leq 2$. 
We can still, however, determine the Euler
characteristic of the generating $\Sigma$-module using
Proposition~\ref{Jaruska_mi_udelala_svickovou!!}. 

Inverting the generating series $g_{\sstildeAss}(t) = t + t^2 + t^3$, 
we read the Euler characteristic of the $\Sigma$-module 
of generators of the minimal model of $\tildeAss$ as
\begin{align*}
&\chi(E(2)) = 1,\ \chi(E(3)) = -1,\ \chi(E(4)) = 0,\ \chi(E(5)) = 4,   
\\
&\chi(E(6)) = -14,\ \chi(E(7)) = 30,\ \chi(E(8)) = -33,\
\chi(E(9)) = -55, \ldots
\end{align*}
The differential $\pa$ of the relevant generators is given by:
\begin{align*}
\partial ( \mu_2) &:= 0,
\\
\partial ( \mu_3) &:= \mu_2 \circ_1 \mu_2 +\mu_2 \circ_2 \mu_2,
\\
\partial ( \mu_5^1)&:=  (\mu_2 \circ_2 \mu_3) \circ_4 \mu_2 - (\mu_3
\circ_3 \mu_2) \circ_4 \mu_2 +(\mu_2 \circ_1 \mu_2) \circ_3 \mu_3 -
(\mu_3 \circ_1 \mu_2 )\circ_3 \mu_2 
\\ &\ +(\mu_2 \circ_1 \mu_3)
\circ_1 \mu_2 - (\mu_3 \circ_1 \mu_2) \circ_1 \mu_2 +(\mu_2 \circ_1
\mu_3) \circ_4 \mu_2 - (\mu_3 \circ_2 \mu_2) \circ_4 \mu_2,
\\
\partial ( \mu_5^2)&:=(\mu_3 \circ_1 \mu_2) \circ_1 \mu_2 - 
(\mu_2 \circ_1 \mu_3) \circ_1 \mu_2 
+(\mu_2 \circ_1 \mu_3) \circ_3 \mu_2 - (\mu_3 \circ_2 \mu_2 )\circ_3
\mu_2 
\\
&\ +(\mu_2 \circ_2 \mu_3) \circ_3 \mu_2 - (\mu_3 \circ_3 \mu_2) \circ_3 \mu_2 
+(\mu_2 \circ_1 \mu_2) \circ_3 \mu_3 - (\mu_3 \circ_1 \mu_2) \circ_4 \mu_2,
\\
\partial ( \mu_5^3)& :=  
(\mu_3 \circ_2 \mu_2) \circ_4 \mu_2 - (\mu_2 \circ_2 \mu_3) \circ_2 \mu_2 
+(\mu_3 \circ_2 \mu_2) \circ_2 \mu_2 - (\mu_2 \circ_1 \mu_2 )\circ_2
\mu_3 
\\
&\ +(\mu_2 \circ_1 \mu_3) \circ_3 \mu_2 - (\mu_2 \circ_1 \mu_3) \circ_1 \mu_2 
+(\mu_2 \circ_1 \mu_2) \circ_1 \mu_3 - (\mu_3 \circ_1 \mu_2) \circ_2
\mu_2,
\mbox { and} 
\\
\partial ( \mu_5^4)&:=  (\mu_3 \circ_1 \mu_2) \circ_3 \mu_2 
- (\mu_3 \circ_3 \mu_2) \circ_3 \mu_2 
+(\mu_2 \circ_2 \mu_3) \circ_3 \mu_2 - (\mu_3 \circ_2 \mu_2 )\circ_3
\mu_2 
\\
&\ +(\mu_2 \circ_1 \mu_2) \circ_2 \mu_3 - (\mu_2 \circ_1 \mu_3) \circ_2 \mu_2
+(\mu_2 \circ_1 \mu_2) \circ_1 \mu_3 - (\mu_2 \circ_1 \mu_3) \circ_1 \mu_2.
\end{align*}
One can make the formulas clearer by using the nested bracket notation.
{}For instance, $\mu_2$ will be represented by $(\q\q)$, $\mu_3$ by
$(\q\q\q)$, $\mu^2_5$ by $(\q\q\q\q\q)^2$, $\mu_3 \circ_2 \mu_2$ by
$(\q(\q\q)\q)$, \&c. With this shorthand, the formulas for the
differential read
\begin{align*}
\pa(\q\q) &:= 0,
\\
\pa(\q\q\q) &:= ((\q\q)\q) + (\q(\q\q)),
\\
\pa(\q\q\q\q\q)^1 &:= (\q(\q\q(\q\q))) - (\q\q(\q(\q\q))) + ((\q\q)(\q\q\q)) -
                    ((\q\q)(\q\q)\q)
\\
&\ + 
(((\q\q)\q\q)\q) - (((\q\q)\q)\q\q) + ((\q\q\q)(\q\q)) - (\q(\q\q)(\q\q)),
\\
\pa(\q\q\q\q\q)^2 &:= (((\q\q)\q)\q\q) - (((\q\q)\q\q)\q) + ((\q\q(\q\q))\q) -
                    (\q(\q(\q\q))\q)
\\
&\ +
(\q(\q(\q\q)\q)) - (\q\q((\q\q)\q)) + ((\q\q)(\q\q\q)) - ((\q\q)\q(\q\q)),
\\
\pa(\q\q\q\q\q)^3 &:= (\q(\q\q)(\q\q)) - (\q((\q\q)\q\q)) + (\q((\q\q)\q)\q) -
                    ((\q(\q\q\q))\q)
\\
&\ +
((\q\q(\q\q))\q) - (((\q\q)\q\q)\q) + (((\q\q\q)\q)\q) - ((\q(\q\q))\q\q),
\mbox { and}
\\
\pa(\q\q\q\q\q)^4 &:= ((\q\q)(\q\q)\q) - (\q\q((\q\q)\q)) + (\q(\q(\q\q)\q)) -
                    (\q(\q(\q\q))\q)
\\
&\ +
((\q(\q\q\q))\q) - ((\q(\q\q)\q)\q) + (((\q\q\q)\q)\q) - (((\q\q)\q\q)\q).
\end{align*}

Let us indicate how we obtained the above formulas. We observed
first that the degree-one subspace $\Gamma(\mu_2,\mu_3)(5)_1 \subset
\Gamma(\mu_2,\mu_3)(5)$ is spanned by $\circ_i$-compositions of two
$\mu_2$'s and one $\mu_3$, i.e., in the bracket language, by nested
bracketings of five $\q$'s with two binary and one ternary
bracket. These elements are in one-to-one correspondence with 
the edges of the
$5$th Stasheff associahedron $K_5$ shown in Figure~\ref{K5},
see~\cite[Section~II.1.6]{markl-shnider-stasheff:book}. 

\begin{figure}[t]
\begin{center}
\setlength{\unitlength}{0.0005in}%
\begin{picture}(3907,3435)(2847,-4288)
\thicklines
\put(3601,-1261){\line( 1,-3){300}}
\put(3901,-2161){\line(-2,-5){200}}
\put(3301,-3661){\line(2,5){300}}
\put(6001,-1261){\line(-1,-3){300}}
\put(5701,-2161){\line( 2,-5){200}}
\put(6301,-3661){\line(-2,5){300}}
\put(4801,-2461){\line(-1,-1){600}}
\put(4201,-3061){\line( 1,-1){600}}
\put(4801,-3661){\line( 1, 1){600}}
\put(5401,-3061){\line(-1, 1){600}}
\put(4801,-2461){\line( 0, 1){1500}}
\put(4801,-961){\line(-4,-1){1200}}
\put(3601,-1261){\line(-1,-2){600}}
\put(3001,-2461){\line( 2,-1){1200}}
\put(5401,-3061){\line( 2, 1){1200}}
\put(6601,-2461){\line(-1, 2){600}}
\put(6001,-1261){\line(-4, 1){1200}}
\put(3001,-2461){\line( 1,-4){300}}
\put(3301,-3661){\line( 5,-2){1500}}
\put(4801,-4261){\line( 0, 1){600}}
\put(3901,-2161){\line( 1, 0){825}}
\put(4876,-2161){\line( 1, 0){825}}
\put(6601,-2461){\line(-1,-4){300}}
\put(6301,-3661){\line(-5,-2){1500}}
\put(4801,-961){\makebox(0,0){$\bullet$}}
\put(6001,-1261){\makebox(0,0){$\bullet$}}
\put(6601,-2461){\makebox(0,0){$\bullet$}}
\put(6301,-3661){\makebox(0,0){$\bullet$}}
\put(4801,-4261){\makebox(0,0){$\bullet$}}
\put(4801,-3661){\makebox(0,0){$\bullet$}}
\put(5401,-3061){\makebox(0,0){$\bullet$}}
\put(4801,-2449){\makebox(0,0){$\bullet$}}
\put(4201,-3061){\makebox(0,0){$\bullet$}}
\put(3601,-1261){\makebox(0,0){$\bullet$}}
\put(3001,-2461){\makebox(0,0){$\bullet$}}
\put(3301,-3661){\makebox(0,0){$\bullet$}}
\put(3901,-2161){\makebox(0,0){$\bullet$}}
\put(5701,-2161){\makebox(0,0){$\bullet$}}
\end{picture}
\end{center}
\caption{Stasheff's associahedron $K_5$.\label{K5}}
\end{figure}
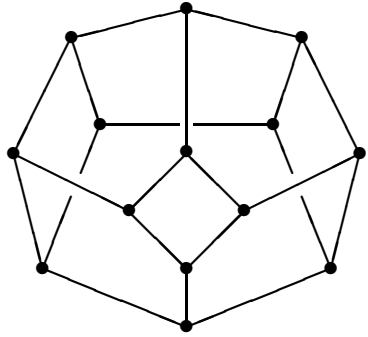

Let $x_e \in \Gamma(\mu_2,\mu_3)(5)_1$ be the element indexed by an
edge $e$ of $K_5$. Clearly $\pa(x_e) = x_a + x_b$, where $a,b$
are the endpoints of $e$ and $x_a, x_b \in \Gamma(\mu_2)(5)_0$ the
elements given by the nested bracketings of five $\q$'s with 
three binary brackets
corresponding to these endpoints. 
We concluded that the $\pa$-cycles in $\Gamma(\mu_2,\mu_3)(5)_1$ are
generated by closed edge-paths of even length in $K_5$; the cycle
corresponding to such a path $P = (e_1,e_2,\ldots,e_{2r})$ being
\[
\sum_{1\leq i \leq 2r}(-1)^{i+1} x_{e_i}.
\]
 
Examples of these paths are provided by two
adjacent pentagons in $K_5$ such as the ones shown in
Figure~\ref{Jarunka}.
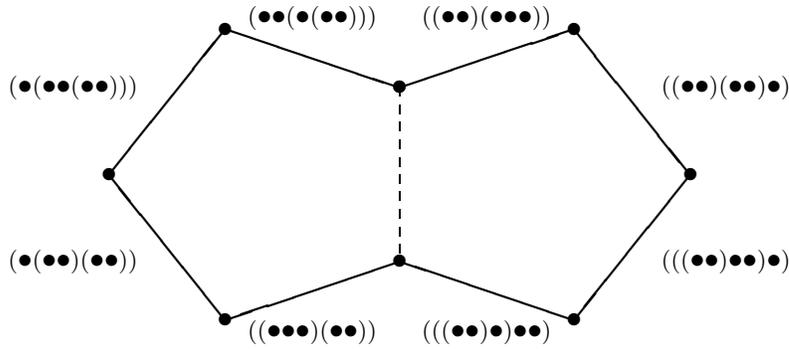
\begin{figure}
{
\unitlength=1.1pt
\begin{picture}(200.00,110.00)(0.00,0.00)
\thicklines
\put(10.00,20.00){\makebox(0.00,0.00)[r]{\scriptsize $(\qq(\qq\qq)(\qq\qq))$}}
\put(70.00,0.00){\makebox(0.00,0.00)[t]{\scriptsize $((\qq\qq\qq)(\qq\qq))$}}
\put(130.00,0.00){\makebox(0.00,0.00)[t]{\scriptsize $(((\qq\qq)\qq)\qq\qq)$}}
\put(190.00,20.00){\makebox(0.00,0.00)[l]{\scriptsize $(((\qq\qq)\qq\qq)\qq)$}}
\put(190.00,80.00){\makebox(0.00,0.00)[l]{\scriptsize $((\qq\qq)(\qq\qq)\qq)$}}
\put(130.00,100){\makebox(0.00,0.00)[b]{\scriptsize $((\qq\qq)(\qq\qq\qq))$}}
\put(70.00,100.00){\makebox(0.00,0.00)[b]{\scriptsize $(\qq\qq(\qq(\qq\qq)))$}}
\put(10.00,80.00){\makebox(0.00,0.00)[r]{\scriptsize $(\qq(\qq\qq(\qq\qq)))$}}
\put(160.00,0.00){\line(-3,1){60.00}}
\put(200.00,50.00){\line(-4,-5){40.00}}
\put(160.00,100.00){\line(4,-5){40.00}}
\put(100.00,80.00){\line(3,1){60.00}}
\put(100.00,80.00){\line(0,1){0.00}}
\put(40.00,100.00){\line(3,-1){60.00}}
\put(0.00,50.00){\line(4,5){40.00}}
\put(40.00,0.00){\line(-4,5){40.00}}
\put(100.00,20.00){\line(-3,-1){60.00}}

\put(160.00,0.00){\makebox(0.00,0.00){$\bullet$}}
\put(200.00,50.00){\makebox(0.00,0.00){$\bullet$}}
\put(160.00,100.00){\makebox(0.00,0.00){$\bullet$}}
\put(100.00,80.00){\makebox(0.00,0.00){$\bullet$}}
\put(100.00,80.00){\makebox(0.00,0.00){$\bullet$}}
\put(40.00,100.00){\makebox(0.00,0.00){$\bullet$}}
\put(0.00,50.00){\makebox(0.00,0.00){$\bullet$}}
\put(40.00,0.00){\makebox(0.00,0.00){$\bullet$}}
\put(100.00,20.00){\makebox(0.00,0.00){$\bullet$}}
\thinlines
\multiput(100,80)(0,-6){10}{\line(0,-1){3}}
\end{picture}}
\caption{\label{Jarunka}
An closed edge path of length $8$ in $K_5$ defining $\pa(\mu_5^1)$.}
\end{figure}
There are also three edge paths of length $4$ given by the three
square faces of $K_5$, but the corresponding cohomology classes have
already been killed by the $\pa$-images of the compositions $\mu_3
\circ_i \mu_3$, $i = 1,2,3$. We showed that there are four
linearly independent edge paths of length $8$ that, together with the
three squares, generate all edge paths of even length in $K_5$. The
generators $\mu_5^1,\mu_5^2,\mu_5^3,\mu_5^4$ correspond to these
paths. 

Also for $a \geq 6$ the $1$-dimensional $\pa$-cycles in 
$\Gamma(\mu_2,\mu_3, \mu_5^1,\mu_5^2,\mu_5^3,\mu_5^4)(a)_1$ are given
by closed edge paths of even length 
in the associahedron $K_a$ but one
can show that they are all generated by the squares and the images
of the paths as in Figure~\ref{Jarunka} under the
face inclusions $K_5 \hookrightarrow K_a$. Therefore 
$(\Gamma(\mu_2,\mu_3, \mu_5^1,\mu_5^2,\mu_5^3,\mu_5^4),\pa)$ is acyclic
in degree $1$, so $\mu_5^1,\mu_5^2,\mu_5^3,\mu_5^4$ are the only
degree two generators of the minimal model of $\tildeAss$.

The construction extends to a minimal model $(\Gamma(E),\pa)$ of the
operad $\tildeAss$ whose differential is {\em not\/} quadratic. It is
simple to show that there does not exists a minimal algebra
$(\Gamma(E'),\pa')$, isomorphic to $(\Gamma(E),\pa)$, with a quadratic
differential. Therefore $\tildeAss$ does not admit a quadratic minimal
model and its non-Koszulness follows not only from the
Ginzburg-Kapranov criterion, but also from
Fact~\ref{sec:dual-quadr-oper-1}.

{}From the above description of the minimal model of $\tildeAss$ one
easily gets the relevant part
\[
\C1 \stackrel{\delta^1}{\longrightarrow} \C2
\stackrel{\delta^2}{\longrightarrow}
\C3 \stackrel{\delta^3}{\longrightarrow}
\C4 \stackrel{\delta^4}{\longrightarrow} \cdots
\]
of the complex defining the deformation cohomology of an
anti-associative algebra $A = (V,\mu)$. One has

-- $\C1 = \Hom(V,V)$

-- $\C2 = \Hom(\otexp V2,V)$

-- $\C3 = \Hom (\otexp V3,V)$, and

-- $\C4 = \Hom(\otexp V5,V) \oplus \Hom(\otexp V5,V) \oplus
   \Hom(\otexp V5,V) \oplus \Hom(\otexp V5,V)$.

\noindent
Observe that $\Cst p = \C p$ for $p = 1,2,3$, while 
$\C4$ consists of $5$-linear maps. The
differential $\delta^p$ agrees with $\dst^p$ for  $p = 1,2$ while,
for $g \in \C3$, one has
\[
\delta^3 (g) = (\delta_1^3(g),\delta_2^3(g),\delta_3^3(g),\delta_4^3(g)),
\]
where
\begin{align*}
\delta_1^3(g)(a,b,c,d,e) &:= 
ag(b,c,de) - g(a,b,c(de)) + (ab)g(c,d,e) - g(ab,cd,e) 
\\
&\
+ g(ab,c,d)e - g((ab)c,d,e) + g(a,b,c)(de) - g(a,bc,de),
\\
\delta_2^3(g)(a,b,c,d,e) &:= 
g((ab)c,d,e) - g(ab,c,d)e + g(a,b,cd)e - g(a,b(cd),e)
\\
&\
+ ag(b,cd,e) - g(a,b,(cd)e) + (ab)g(c,d,e) - g(ab,c,de),
\\
\delta_3^3(g)(a,b,c,d,e) &:= 
g(a,bc,de) - ag(bc,d,e) + g(a,(bc)d,e) - a(g(b,c,d)e)
\\
&\
+ g(a,b,cd)e - g(ab,c,d)e + (g(a,b,c)d)e - g(a(bc),d,e), \mbox { and}
\\
\delta_4^3(g)(a,b,c,d,e) &:= 
g(ab,cd,e) - g(a,b,(cd)e) + a g(b,cd,e) - g(a,b(cd),e)
\\
&\
+ (ag(b,c,d))e - g(a,bc,d)e + (g(a,b,c)d)e - g(ab,c,d)e,
\end{align*}
for $a,b,c,d,e \in V$.  The following proposition follows
from~\cite[Section~4]{markl:JPAA96}.

\begin{proposition}
The cohomology $\H*$ governs deformations of
anti-associative algebras. This means that $\H2$ parametrizes isomorphism
classes of infinitesimal deformations and $\H3$ contains obstructions
to extensions of partial deformations.
\end{proposition}

\section{Free partially associative $n$-algebras}
\label{Posledni_den_v_Mulhouse.}

In~\cite{gnedbaye:CM96}, A.V.\ Gnedbaye described free degree $d$
partially associative $n$-algebras in the situations when
$d=0$ and $n$ was even.  In this section we extend Gnedbaye's
description of free $\pAss^n_d$-algebras to all cases when $d$ and $n$
have the same parity.

Let $\Free$ be the free $\pAss^n_d$-algebra generated by
a graded vector space $V$. It obviously decomposes as
\[
\Free = \bigoplus_{l \geq 0} \Free_l,
\] 
where $\Free_l \subset \Free$ is the subspace  generated by elements
obtained by applying the structure $n$-ary 
multiplication $\mu$ to elements of $V$ $l$-times. For instance,
$\Free_0 \cong V$ and $\Free_1 \cong \otexp Vn$. 

Denote by $\Tree l$, $l \geq 1$, the set of planar directed (= rooted)
trees with $l(n-1)+1$ leaves whose vertices have precisely $n$
incoming edges (see~\cite[Section~4]{markl:handbook} 
or~\cite[II.1.5]{markl-shnider-stasheff:book} for terminology). 
We extend the definition to $l = 0$ by putting $\Tree 0 :=
\{\exepttree\}$, the one-point set consisting of the exceptional tree
with one leg and no internal vertex. Clearly, each tree in $\Tree l$
has exactly $l$ vertices. For each $l$ there is a natural epimorphism
\begin{equation}
\label{epp}
\omega: \Tree l \times \otexp V{l(n-1)+1}  \epi \Free_l
\end{equation}
given by interpreting the trees in $\Tree l$ as the `pasting schemes'
for the iterated multiplication~$\mu$. More precisely, if  $T \in
\Tree l$ and $\Rada v1{l(n-1)+1} \in V$, then
\[
\omega(T \times (\Rada v1{l(n-1)+1}))\in \Free_l
\] 
is obtained by
decorating the vertices of $T$ by $\mu$,
the leaves of $T$ by elements $\Rada v1{l(n-1)+1}$, 
and performing the indicated composition, observing the Koszul 
sign rule in the nontrivially graded
cases. 

Let $\Sree l \subset \Tree l$ be the subset consisting of trees having
the property that the leftmost incoming edge of each vertex is a
leaf. Since these trees correspond to the generators of partially
associative algebras considered by Gnedbaye in~\cite{gnedbaye:CM96},
we call them {\em Gnedbaye's trees\/}. Therefore $\Sree 0 = \Tree 0 =
\{\exepttree\}$, $\Sree 1$ is the one-point set consisting of the
$n$-corolla
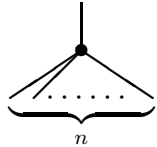
\begin{center}
\unitlength=.900000pt
\begin{picture}(60.00,40.00)(0.00,30.00)
\thicklines
\put(24.00,30.00){\makebox(0.00,0.00){$\cdots$}}
\put(42.00,30.00){\makebox(0.00,0.00){$\cdots$}}
\put(30.00,50.00){\makebox(0.00,0.00){$\bullet$}}
\put(30.00,27.00){\makebox(0.00,0.00)[t]
           {$\underbrace{\rule{55pt}{0pt}}_n$}}
\put(30.00,50.00){\line(3,-2){30.00}}
\put(30.00,50.00){\line(-1,-1){20.00}}
\put(30.00,50.00){\line(-3,-2){30.00}}
\put(30.00,50.00){\line(0,1){20.00}}
\end{picture}
\end{center}
and $\Sree 2$ has $n-1$ elements  
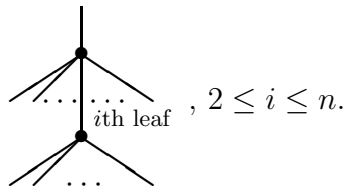
\begin{center}
\unitlength=.900000pt
\begin{picture}(60.00,70.00)(0.00,0.00)
\thicklines
\put(35.00,23.00){\makebox(0.00,0.00)[l]{\scriptsize $i$th leaf}}
\put(32.00,-5.00){\makebox(0.00,0.00){$\cdots$}}
\put(22.00,30.00){\makebox(0.00,0.00){$\cdots$}}
\put(42.00,30.00){\makebox(0.00,0.00){$\cdots$}}
\put(30.00,15.00){\makebox(0.00,0.00){$\bullet$}}
\put(30.00,50.00){\makebox(0.00,0.00){$\bullet$}}
\put(30.00,15.00){\line(3,-2){30.00}}
\put(30.00,15.00){\line(-1,-1){20.00}}
\put(30.00,15.00){\line(-3,-2){30.00}}
\put(30.00,50.00){\line(0,-1){35.00}}
\put(30.00,50.00){\line(3,-2){30.00}}
\put(30.00,50.00){\line(-1,-1){20.00}}
\put(30.00,50.00){\line(-3,-2){30.00}}
\put(30.00,50.00){\line(0,1){20.00}}
\put(75.00,30.00){\makebox(0.00,0.00)[l]{,~$2\leq i \leq n$.}}
\end{picture} 
\end{center}
 It is clear that, for $l \geq 3$, $\Sree l$ consists of trees of the form
\begin{equation}
\label{JarunkA}
\raisebox{-30pt}{\rule{0pt}{0pt}}
\unitlength=1pt
\begin{picture}(60.00,30.00)(0.00,40.00)
\thicklines
\put(22.00,30.00){\makebox(0.00,0.00){$\cdots$}}
\put(42.00,30.00){\makebox(0.00,0.00){$\cdots$}}
\put(30.00,50.00){\makebox(0.00,0.00){$\bullet$}}
\put(30.00,50.00){\line(0,-1){20.00}}
\put(30.00,50.00){\line(3,-2){30.00}}
\put(30.00,50.00){\line(-1,-1){20.00}}
\put(30.00,50.00){\line(-3,-2){50.00}}
\put(30.00,50.00){\line(0,1){20.00}}
\put(25.00,30.00){\makebox(0.00,0.00)[t]{\gl{S_2}}}
\put(45.00,30.00){\makebox(0.00,0.00)[t]{\gl{S_i}}}
\put(75.00,30.00){\makebox(0.00,0.00)[t]{\gl{S_n}}}
\end{picture}
\end{equation}
where $S_i \in \Sree {l_i}$ for $2 \leq i \leq n$ and $l_2 + \cdots +
l_n = l-1$.

As we already mentioned at the beginning of this section, Gnedbaye
described, in \cite[Proposition~12]{gnedbaye:CM96}, free degree $d$
partially associative $n$-algebras for $d=0$ and $n$ even. We extend
his result to the cases where $n$ and $d$ are of the same parity:

\begin{theorem}
\label{jarkA}
Assume that $n$ and $d$ are of the same parity. Then the
restriction (denoted by the same symbol)
\begin{equation}
\label{jarKa}
\omega : \Sree l \times \otexp V{l(n-1)+1}  \longrightarrow \Free_l
\end{equation}
of the epimorphism~(\ref{epp}) is an isomorphism, for each $l \geq 0$.
\end{theorem}

Observe that, if the parities of $d$ and $n$ are as in the statement, 
the operad $\pAss^n_d$ is Koszul by Theorem~\ref{2}.

\begin{proof}[Proof of Theorem~\ref{jarkA}]
Axiom~(\ref{Jaruska_sibalsky_mrka.}) for partially associative algebras
implies that each iterated multiplication in $\Free_l$ can be 
brought into a linear combination of multiplications 
with the pasting schemes in $\Sree l$, 
i.e.~that the map~(\ref{jarKa}) is an epimorphism.
Let us prove this statement by induction. 

There is nothing to prove for $l=0,1$. Assume that we have established
the claim for all $0\leq l \leq k$, $k \geq 1$, and prove it for
$l=k$. Let $\mu_T$ be an iterated multiplication with the pasting
scheme $T \in \Tree k$. There are two possibilities. The {\em first
case:\/} the tree $T$ is of the form
\begin{equation}
\label{Dnes_na_obed_s_Jaruskou.}
\raisebox{-30pt}{\rule{0pt}{0pt}}
\unitlength=1pt
\begin{picture}(60.00,30.00)(0.00,40.00)
\thicklines
\put(22.00,30.00){\makebox(0.00,0.00){$\cdots$}}
\put(42.00,30.00){\makebox(0.00,0.00){$\cdots$}}
\put(30.00,50.00){\makebox(0.00,0.00){$\bullet$}}
\put(30.00,50.00){\line(0,-1){20.00}}
\put(30.00,50.00){\line(3,-2){30.00}}
\put(30.00,50.00){\line(-1,-1){20.00}}
\put(30.00,50.00){\line(-3,-2){50.00}}
\put(30.00,50.00){\line(0,1){20.00}}
\put(25.00,30.00){\makebox(0.00,0.00)[t]{\gl{T_2}}}
\put(45.00,30.00){\makebox(0.00,0.00)[t]{\gl{T_i}}}
\put(75.00,30.00){\makebox(0.00,0.00)[t]{\gl{T_n}}}
\end{picture}
\end{equation}
for some $T_i \in \Tree {l_i}$, $2 \leq i \leq n$, with $l_2 + \cdots
+ l_n  = k-1$. Then 
\[
\mu_T = \mu(\id \ot \mu_{T_2} \ot \cdots \ot \mu_{T_n}),
\]  
where $\mu_{T_i}$ denotes the iterated multiplication with the pasting
scheme $T_i$. By induction, each $\rada{\mu_{T_2}}{\mu_{T_n}}$
is a linear combination of iterated multiplications whose pasting
schemes belong to the subsets $\rada{\Sree {l_2}}{\Sree {l_n}}$,
respectively. The observation that the
tree~(\ref{Dnes_na_obed_s_Jaruskou.}) belongs to $\Sree k$ if
$T_i \in \Sree {l_i}$ for each $2 \leq i \leq n$ 
completes the induction step for this case.

In the {\em second case,\/} the tree $T$ has the form
\[
\raisebox{-30pt}{\rule{0pt}{0pt}}
\unitlength=1pt
\begin{picture}(60.00,30.00)(0.00,40.00)
\thicklines
\put(17.00,30.00){\makebox(0.00,0.00){$\cdots$}}
\put(47.00,30.00){\makebox(0.00,0.00){$\cdots$}}
\put(30.00,50.00){\makebox(0.00,0.00){$\bullet$}}
\put(30.00,50.00){\line(0,-1){20.00}}
\put(30.00,50.00){\line(3,-2){30.00}}
\put(30.00,50.00){\line(-3,-2){30.00}}
\put(30.00,50.00){\line(0,1){20.00}}
\put(15.00,30.00){\makebox(0.00,0.00)[t]{\gl{T_1}}}
\put(45.00,30.00){\makebox(0.00,0.00)[t]{\gl{T_i}}}
\put(75.00,30.00){\makebox(0.00,0.00)[t]{\gl{T_n}}}
\end{picture}
\]
where $T_i \in \Tree {l_i}$ for $1 \leq i \leq n$, $l_1 + \cdots + l_n
= k-1$ and $l_1 \geq 1$. Now
\[
\mu_T = \mu(\mu_{T_1} \ot \cdots \ot \mu_{T_n})
\]
and we may assume, by induction, that $T_i \in \Sree {l_i}$ for each
$1 \leq i \leq n$. In particular, $T_1$ is as in~(\ref{JarunkA}) with
$S_j \in \Sree {l'_j}$, $2\leq j \leq n$ such that $l'_2 + \cdots +l'_n =
l_1-1$, thus
\[
\mu_T = \mu(\mu \ot \otexp {\id}{n-1})(\id \ot \mu_{S_2}\ot \cdots \ot
\mu_{S_n} \ot \mu_{T_2} \ot \cdots \ot \mu_{T_n}).
\]
By~(\ref{Jaruska_sibalsky_mrka.}), one may replace the factor
$\mu(\mu \ot \otexp {\id}{n-1})$ by the linear combination
\[
-\sum_{i=2}^n (-1)^{(i+1)(n-1)}
\mu\left(\id^{\otimes i-1} \otimes \mu \otimes \id^{\otimes n-i}\right)
\]
which brings $\mu_T$ also in the second case to the
desired form and finishes the induction step.

To prove that the map~(\ref{jarKa}) it is an isomorphism, it suffices now
to compare the dimensions of $\Sree l \times \otexp V{l(n-1)+1}$ and
$\Free_l$. It follows from the
description~\cite[Proposition~II.1.25]{markl-shnider-stasheff:book} of
the free operad algebra that, for each $l \geq 0$,
\[
\Free_l \cong \pAss^n_d(l(n-1)+1) \otimes{\Sigma_{l(n-1)+1}}\otexp V{l(n-1)+1}.
\]  
Theorem~\ref{jarkA} will thus be established if we prove that 
\[
S^n_l := \card(\Sree l) \mbox { equals } A^n_l :=
\frac 1{(l(n-1)+1)!}    \dim(\pAss^n_d(l(n-1)+1)),
\]
for each $l \geq 0$. It easily follows from~(\ref{JarunkA}) 
that the sequence $\{S^n_l\}_{l \geq 0}$ is
determined by the recursion $S^n_0 = 1$ and
\begin{equation}
\label{Pusinka_Jarka}
S^n_l:=\sum_{\substack{0 \leq l_2 , \cdots ,l_n\leq l-1 
\\ l_2 +\cdots+ l_{n-1}= l-1}} S^n_{l_2} \cdots S^n_{l_n} \ \mbox{for} \ 
l \geq 1.
\end{equation}
In Proposition~\ref{Za_tri_dny_uz_domu} below, which is of independent
interest,  we prove that the sequence
$\{A^n_l\}_{l \geq 0}$ satisfies the same recursion. This finishes the proof.
\end{proof}

Recursion~(\ref{Pusinka_Jarka}) appeared, with $p^{< n-1 >}_l$
in place of $S^n_l$, in~\cite[Section~3.4]{gnedbaye:CM96}.
Theorem~\ref{jarkA} gives a realization of free $\pAss^n_d$-algebras 
in the Koszul case ($n \equiv d$ mod $2$) by putting
\[
\Free :=  \bigoplus_{l \geq 0}\Sree l \times \otexp V{l(n-1)+1}.
\] 
We leave as an exercise to describe the structure $n$-ary
multiplication of $\Free$ in this
language, see~\cite{gnedbaye:CM96}.

\begin{proposition}
\label{Za_tri_dny_uz_domu}
The Poincar\'e series of the operad $\pAss^n_d$ is, in the Koszul case 
(with $n$ and $d$ of the same parity), given by
\begin{equation} 
\label{Pusinka_Jarka!}
g_{\sspAss^n_d}(t)=\sum_{l\geq0} (-1)^{ln}A^n_lt^{l(n-1)+1}
\end{equation}
where the coefficients $\{A^n_l\}_{l\geq 0}$ are defined recursively by
$A^n_0:=1$ and
\begin{equation}
\label{JarKa}
A^n_l:=\sum_{\substack{0 \leq l_2 , \cdots ,l_n\leq l-1 
\\ l_2 +\cdots+ l_{n-1}= l-1}} A^n_{l_2} \cdots A^n_{l_n} \ \mbox{for} \ 
l \geq 1.
\end{equation}
\end{proposition}

\begin{proof}
One can easily check that the recursive definition~(\ref{JarKa}) of the
coefficients of $f(t) : = g_{\sspAss^n_d}(t)$ is
equivalent to the functional equation
\[
f(t) = t\left(1 + (-1)^n  f(t)^{n-1}\right)
\] 
which in turn immediately implies that $f(t)$ is the unique formal solution of
\[ 
g_{\sstAss^n_{-d+ n-2}}(-f(-t))=t,
\] 
where the Poincar\'e series $g_{\sstAss^n_{-d+ n-2}}(t)$ is as in the first
line of~(\ref{Zitra_ma_priletet_Jarka_do_Mulhouse}) because $-d+ n-2$
is even. Since we are in the Koszul case, the above display means, by
Theorem~\ref{gk}, that $f(t)$ is the Poincar\'e series of 
$(\tAss^n_{-d+ n-2})^! = \pAss^n_d$. This proves the proposition.
\end{proof}

The description of the Poincar\'e  
series of $\pAss^n_d$ for $n$ and $d$ of the same parity given in
Proposition~\ref{Za_tri_dny_uz_domu} implies that
the Poincar\'e series of $\ptildeAss^n_d$ for $d$
odd equals
\[ 
g_{\ssptildeAss^n_d}(t)=\sum_{l\geq0}(-1)^{l}A^n_lt^{l(n-1)+1},
\]
with $\{A^n_l\}_{l \geq 0}$ having the meaning as in~(\ref{Pusinka_Jarka!}).

\begin{example}{\rm
Using {\tt Matematica\/}, we calculated initial values of the
series $\{A^3_l\}_{l \geq 0}$
as $A^3_0 = 1$, $A^3_1 = 1$, $A^3_2 = 2$, $A^3_3 =
5$, $A^3_4 = 14$, $A^3_5 = 42$, $A^3_6 = 132$, $A^3_7 = 429$, $A^3_8 = 1~430$,
$A^3_9 = 4~862$, $A^3_{10} = 16~796$, \&c. 
}
\end{example}

\begin{remark}
{\rm
\label{sec:free-part-assoc}
If $n$
and $d$ are of different parities, the map~(\ref{jarKa}) of
Theorem~\ref{jarkA}, while always being an epimorphism,
need not be a monomorphism. This means that there may be ``unexpected
relations'' in the free algebra $\Free$. Consequently,
the vector space
\[
\bigoplus_{l \geq 1}\Sree l \times \otexp V{l(n-1)+1}
\]
associated to Gnedbaye's trees {\em cannot be\/}  
equipped with a structure of partially associative algebra
generated by $V = \Sree 1 \times V$.
For instance, while  the dimension of $S^3_3$
equals $5$, the dimension of $\pAss ^3_0(7)$ equals $7! \cdot 4$, so
$\pAss ^3_0(V)_3$ has one copy of $\otexp V7$ less than $\Sreep 3
\times V^{\otimes 7}$.
More concretely, it turns out that
\[
v_1(v_2v_3(v_4v_5v_6))v_7+
v_1v_2(v_3(v_4v_5v_6)v_7)+
v_1v_2(v_3v_4(v_5v_6v_7))=0, \ \mbox { for }\ \Rada v17 \in V,
\]
in the free algebra $\pAss ^3_0(V)$ and therefore also in every
degree-$0$ partially associative $3$-algebra. In terms of
Gnedbaye's trees,
\begin{equation}
\label{netece_voda}
\raisebox{-27pt}{{\rule{0pt}{0pt}}}
{
\unitlength=1.000000pt
\begin{picture}(170.00,20.00)(0.00,20)
\thicklines
\put(170.00,20.00){\makebox(0.00,0.00){$=  \hskip .3em 0$.}}
\put(100.00,20.00){\makebox(0.00,0.00){$+$}}
\put(40.00,20.00){\makebox(0.00,0.00){$+$}}
\put(150.00,10.00){\makebox(0.00,0.00){$\bullet$}}
\put(140.00,20.00){\makebox(0.00,0.00){$\bullet$}}
\put(130.00,30.00){\makebox(0.00,0.00){$\bullet$}}
\put(80.00,10.00){\makebox(0.00,0.00){$\bullet$}}
\put(80.00,20.00){\makebox(0.00,0.00){$\bullet$}}
\put(70.00,30.00){\makebox(0.00,0.00){$\bullet$}}
\put(20.00,10.00){\makebox(0.00,0.00){$\bullet$}}
\put(10.00,20.00){\makebox(0.00,0.00){$\bullet$}}
\put(10.00,30.00){\makebox(0.00,0.00){$\bullet$}}
\put(150.00,10.00){\line(0,-1){10.00}}
\put(150.00,10.00){\line(-1,-1){10.00}}
\put(150.00,10.00){\line(1,-1){10.00}}
\put(140.00,20.00){\line(0,-1){10.00}}
\put(140.00,20.00){\line(-1,-1){10.00}}
\put(140.00,20.00){\line(1,-1){10.00}}
\put(130.00,40.00){\line(0,-1){20.00}}
\put(130.00,30.00){\line(1,-1){10.00}}
\put(120.00,20.00){\line(1,1){10.00}}
\put(80.00,10.00){\line(1,-1){10.00}}
\put(80.00,10.00){\line(-1,-1){10.00}}
\put(80.00,20.00){\line(0,-1){20.00}}
\put(80.00,20.00){\line(1,-1){10.00}}
\put(70.00,10.00){\line(1,1){10.00}}
\put(70.00,40.00){\line(0,-1){20.00}}
\put(70.00,30.00){\line(1,-1){10.00}}
\put(60.00,20.00){\line(1,1){10.00}}
\put(20.00,10.00){\line(-1,-1){10.00}}
\put(20.00,10.00){\line(0,-1){10.00}}
\put(10.00,20.00){\line(1,-1){20.00}}
\put(10.00,20.00){\line(0,-1){10.00}}
\put(10.00,20.00){\line(-1,-1){10.00}}
\put(10.00,30.00){\line(1,-1){10.00}}
\put(10.00,30.00){\line(-1,-1){10.00}}
\put(10.00,40.00){\line(0,-1){20.00}}
\end{picture}}  
\end{equation}

This relation can be read off the corresponding galgal
in~(\ref{G3}). To do so, decorate the vertices of $G^3$ by $+$ or $-$
as 
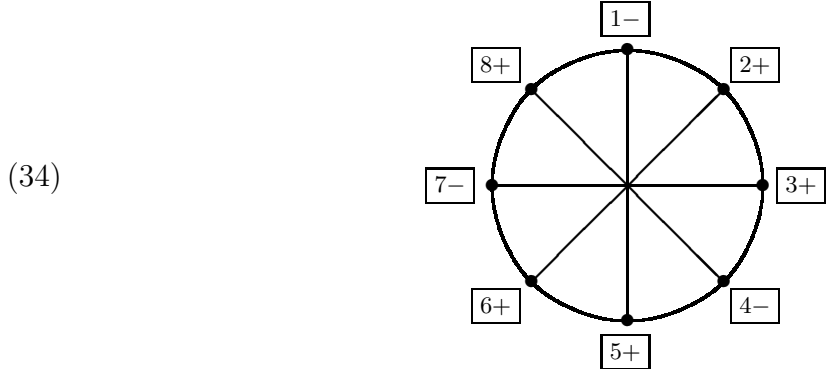
\begin{equation}
\label{Jarca_ma_dovcu.}
\raisebox{-2.3cm}{\rule{0pt}{0pt}}
\unitlength=.9cm
\begin{picture}(6,2.6)(-3.00,0)
\thicklines
\put(0,0){\hustykrouzek}    
\put(0,2){\makebox(0,0){$\bullet$}}
\put(0,2.2){\makebox(0,0)[b]{\scriptsize $\R{1-}$}}
\put(0,2){\line(0,-1){4}}
\put(-1.42,-1.42){\line(1,1){2.84}}
\put(-1.42,1.42){\line(1,-1){2.84}}
\put(-2,0){\line(1,0){4}}
\put(2,0){\makebox(0,0){$\bullet$}}
\put(2.2,0){\makebox(0,0)[l]{\scriptsize $\R{3+}$}}
\put(0,-2){\makebox(0,0){$\bullet$}}
\put(0,-2.2){\makebox(0,0)[t]{\scriptsize $\R{5+}$}}
\put(-2,0){\makebox(0,0){$\bullet$}}
\put(-2.2,0){\makebox(0,0)[r]{\scriptsize $\R{7-}$}}
\put(1.42,1.42){\makebox(0,0){$\bullet$}}
\put(1.52,1.52){\makebox(0,0)[lb]{\scriptsize $\R{2+}$}}
\put(-1.42,1.42){\makebox(0,0){$\bullet$}}
\put(-1.52,1.52){\makebox(0,0)[rb]{\scriptsize $\R{8+}$}}
\put(1.42,-1.42){\makebox(0,0){$\bullet$}}
\put(1.52,-1.52){\makebox(0,0)[lt]{\scriptsize $\R{4-}$}}
\put(-1.42,-1.42){\makebox(0,0){$\bullet$}}
\put(-1.52,-1.52){\makebox(0,0)[tr]{\scriptsize $\R{6+}$}}
\end{picture}
\end{equation}
and take the sum of the corresponding elementary relations, with the
above choice of signs. Notice that endpoints of all edges
in~(\ref{Jarca_ma_dovcu.}) differ by sign, except for the edges 
\[
\R2-\R3,\ \R2- \R6 \ \mbox  { and}\  \ \R5- \R6.
\]
The sum of the elements of the free algebra corresponding to
these edges must be zero. It is is represented by Gnedbaye's trees at the left hand
side of~(\ref{netece_voda}).

We conclude that the map~(\ref{jarKa}) has, for $n=3$, $d=0$ and
$l = 3$, a nontrivial kernel. The Poincar\'e series of $\pAss^3_0$
was calculated in~\cite{goze-remm:n-ary} as
\[
g_{\sspAss^3_0}(t)=t+t^3+2t^5+4t^7+5t^9+6t^{11}+7t^{13}+8t^{15}\cdots.
\]

The same phenomenon takes place also for $n=5$. By choosing appropriate
decorations of the vertices of the $5$th galgal $G^5$ depicted in 
Figure~\ref{G5}, one can verify that the equation
\begin{align*}
0=&\ v_{1}(v_{2}v_{3}v_{4}v_{5}(v_{6}v_{7}v_{8}v_{9}v_{10})) v_{11}v_{12}v_{13}+ 
v_{1}v_{2}(v_{3}v_{4}v_{5}(v_{6}v_{7}v_{8}v_{9}v_{10})v_{11})v_{12}v_{13}
\\
&+ 
v_{1}v_{2}(v_{3}v_{4}v_{5}v_{6}(v_{7}v_{8}v_{9}v_{10}v_{11}))v_{12}v_{13}+ 
v_{1}v_{2}v_{3}(v_{4}v_{5}(v_{6}v_{7}v_{8}v_{9}v_{10})v_{11}v_{12})v_{13}
\\
&+ 
v_{1}v_{2}v_{3}(v_{4}v_{5}v_{6}(v_{7}v_{8}v_{9}v_{10}v_{11})v_{12})v_{13}+ 
v_{1}v_{2}v_{3}(v_{4}v_{5}v_{6}v_{7}(v_{8}v_{9}v_{10}v_{11}v_{12}))v_{13} 
\\
&+
v_{1}v_{2}v_{3}v_{4}(v_{5}(v_{6}v_{7}v_{8}v_{9}v_{10})v_{11}v_{12}v_{13})+ 
v_{1}v_{2}v_{3}v_{4}(v_{5}v_{6}(v_{7}v_{8}v_{9}v_{10}v_{11})v_{12}v_{13})
\\
&+ 
v_{1}v_{2}v_{3}v_{4}(v_{5}v_{6}v_{7}(v_{8}v_{9}v_{10}v_{11}v_{12})v_{13})+
v_{1}v_{2}v_{3}v_{4}(v_{5}v_{6}v_{7}v_{8}(v_{9}v_{10}v_{11}v_{12}v_{13}))
\end{align*}
holds for elements $\Rada v1{13}$ of  
any degree-$0$ partially associative $5$-ary algebra. In
terms of Gnedbaye's trees, the right hand side is represented by the
sum
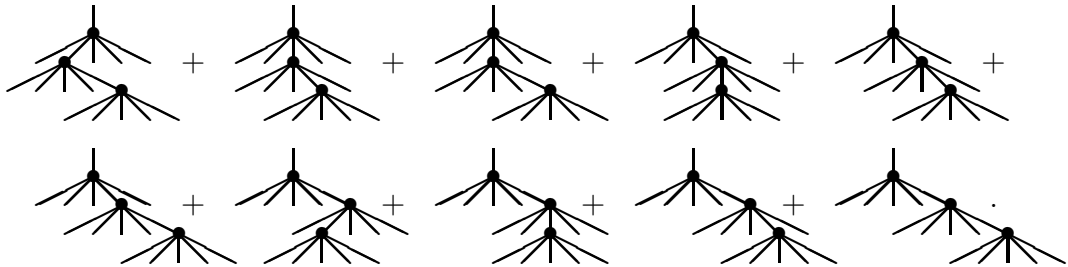
\begin{center}
\unitlength=.38cm
\begin{picture}(4.00,8.50)(0.00,2.5)
\thicklines
\put(-14,10){\jaruska-1|2|+} \put(-7,10){\jaruska0|1|+}
\put(0,10){\jaruska0|2|+} \put(7,10){\jaruska1|0|+}
\put(14,10){\jaruska1|1|{+}}
\put(-14,5){\jaruska1|2|+} \put(-7,5){\jaruska2|-1|+}
\put(0,5){\jaruska2|0|+} \put(7,5){\jaruska2|1|+} \put(14,5){\jaruska2|2|{.}}
\end{picture}
\end{center}
We believe that the same explicit calculations can be performed for
degree $0$ partially associative $n$-algebras with an arbitrary odd $n$.  
}\end{remark}

\section{Open problems}
\label{open}

The first question which our paper leaves open is the 
Koszulness of the operads $\tAss^n_d$ with $d$ odd and $n \geq 8$. The
method used in the proof of Proposition~\ref{5} does not apply to
these cases and indeed, our numerical tests mentioned in
Remark~\ref{JARuska} suggest that it may happen that all coefficients in the
formal inverse of $t-t^{n}+t^{2n-1}$  are non-negative.

Even if this happens, it would not necessarily mean that the operad $\tAss^n_d$
is Koszul, only that subtler methods must be applied to
that case. For instance, one may try to compare the coefficients of
this formal inverse to the dimensions of the components of the
dual operad $(\tAss^n_d)^!$.

Understanding these components is, of course, equivalent to finding a
basis for the free partially associative algebras in the non-Koszul
cases. This problem was solved, in~\cite{goze-remm:n-ary}, for free
$\pAss^3_0$-algebras; for $n \geq 4$ it remains open.

The last problem we want to formulate here is to find more
about the minimal model of the anti-associative operads $\tildeAss$, 
or even to describe it completely. As far
as we know, beyond the `obvious' cases, 
no complete description of the minimal model of a
non-Koszul operad is known. Since $\tildeAss$ is one of the simplest non-Koszul
operads, it is the first obvious candidate to attack.
A related task is to find as much as information about 
minimal models of the remaining non-Koszul $n$-ary operads as possible.


\def\cprime{$'$}

\end{document}